\documentclass[english,11pt, a4paper]{article}

\usepackage[T1]{fontenc}

\usepackage{amssymb}
\usepackage{amsmath}
\usepackage{amsfonts}
\usepackage{amsthm}
\usepackage{mathtools}
\usepackage{mathabx}
\usepackage{comment}
\usepackage[dvipsnames]{xcolor}
\usepackage{graphicx}
\usepackage{float}

\usepackage{mdframed}
\usepackage{enumitem} 

\usepackage[export]{adjustbox}
\usepackage{bbold}
\usepackage[percent]{overpic}
\usepackage{units}

\usepackage{booktabs,array}
\usepackage{url} 
\usepackage{hyperref}

\let\OLDthebibliography\thebibliography
\renewcommand\thebibliography[1]{
  \OLDthebibliography{#1}
  \setlength{\parskip}{0pt}
  \setlength{\itemsep}{0pt plus 0.3ex}
}

\hypersetup{
    colorlinks=true,
    citecolor=purple!70!blue,
    linkcolor=black,
    filecolor=black,      
    urlcolor=black,
    pdftitle={Free energy and quark potential in Ising lattice gauge theory via Cluster Expansions},
    }

\usepackage{pgfplots}
\pgfplotsset{compat=1.6}

\usetikzlibrary{calc,quotes,angles,arrows.meta,positioning, shapes}
\usetikzlibrary{decorations.pathreplacing,decorations.markings}

\usepackage{float}

\usepackage{caption}
\usepackage{subcaption}

\usepackage[nocompress]{cite}

\theoremstyle{plain}%
\newtheorem{theorem}{Theorem}[section]
\newtheorem{lemma}[theorem]{Lemma}
\newtheorem{proposition}[theorem]{Proposition}

\newtheorem*{conjecture*}{Conjecture}

 \numberwithin{equation}{section}

\theoremstyle{definition}
\newtheorem{definition}[theorem]{Definition}

\theoremstyle{remark}
\newtheorem{remark}[theorem]{Remark}

\let \le \leqslant
 \let \leq \leqslant
 \let \geq \geqslant
 \let \ge \geqslant

\DeclareMathOperator{\real}{Re}

\DeclareMathOperator{\support}{supp}

\DeclareMathOperator{\dist}{dist}

\usepackage{fancyhdr}
\usepackage{graphics}
\usepackage{graphicx}
\usepackage{soul}

\definecolor{detailcolor00}{rgb}{0.4405, 0.204, 0.343}
\definecolor{detailcolor01}{rgb}{0.546, 0.215, 0.352}
\definecolor{detailcolor02}{rgb}{0.675, 0.247, 0.387} 
\definecolor{detailcolor03}{rgb}{0.775, 0.317, 0.455}
\definecolor{detailcolor04}{rgb}{0.830, 0.421, 0.553} 
\definecolor{detailcolor05}{rgb}{0.831, 0.533, 0.663}
\definecolor{detailcolor06}{rgb}{0.779, 0.619, 0.775}
\definecolor{detailcolor07}{rgb}{0.724, 0.694, 0.827}
\definecolor{detailcolor08}{rgb}{0.687, 0.770, 0.880}
\definecolor{detailcolor09}{rgb}{0.671, 0.839, 0.904}
\definecolor{detailcolor10}{rgb}{0.659, 0.872, 0.882}

\usepackage{scalerel,stackengine}
\newcommand\pig[1]{\scalerel*[5.5pt]{\Big#1}{%
  \ensurestackMath{\addstackgap[1.5pt]{\big#1}}}}
\newcommand\pigl[1]{\mathopen{\pig{#1}}}
\newcommand\pigr[1]{\mathclose{\pig{#1}}}

\ifdim\overfullrule>0pt 
  \usepackage{environ}

  \NewEnviron{tikzpicture}{%
    \begin{pgfpicture}
    \pgfpathrectanglecorners{\pgfpointorigin}{\pgfpoint{3cm}{3cm}}%
     \pgfusepath{stroke}\end{pgfpicture}%
  }
\fi

\makeatletter
\newcommand{\hathat}[1]{%
\begingroup%
  \let\macc@kerna\z@%
  \let\macc@kernb\z@%
  \let\macc@nucleus\@empty%
  \hat{\raisebox{.3ex}{\vphantom{\ensuremath{#1}}}\smash{\hat{#1}}}%
\endgroup%
}

\newcommand{\smallhathat}[1]{%
\begingroup%
  \let\macc@kerna\z@%
  \let\macc@kernb\z@%
  \let\macc@nucleus\@empty%
  \hat{\raisebox{.05ex}{\vphantom{\ensuremath{#1}}}\smash{\hat{#1}}}%
\endgroup%
}

\newcommand{\smallsmallhathat}[1]{%
\begingroup%
  \let\macc@kerna\z@%
  \let\macc@kernb\z@%
  \let\macc@nucleus\@empty%
  \hat{\raisebox{-.2ex}{\vphantom{\ensuremath{#1}}}\smash{\hat{#1}}}%
\endgroup%
}
\makeatother

\makeatletter
\newcommand{\subalign}[1]{%
  \vcenter{%
    \Let@ \restore@math@cr \default@tag
    \baselineskip\fontdimen10 \scriptfont\tw@
    \advance\baselineskip\fontdimen12 \scriptfont\tw@
    \lineskip\thr@@\fontdimen8 \scriptfont\thr@@
    \lineskiplimit\lineskip
    \ialign{\hfil$\m@th\scriptstyle##$&$\m@th\scriptstyle{}##$\hfil\crcr
      #1\crcr
    }%
  }%
}
\makeatother

\title{Free energy and quark potential in \\ Ising lattice gauge theory via cluster expansion}

\author{Malin P. Forsstr\"om\thanks{University of Gothenburg, email: palo@chalmers.se} \, and Fredrik Viklund\thanks{KTH Royal Institute of Technology, email: frejo@kth.se}} 

\begin{document}

\maketitle

\begin{abstract}
    We revisit the cluster expansion for Ising lattice gauge theory on $\mathbb{Z}^m, \, m \ge 3,$ with Wilson action, 
    at a fixed inverse temperature \( \beta\) in the low-temperature regime.
    We prove existence and analyticity of the 
    infinite volume limit of the free energy 
    and compute the first few terms in its expansion in powers of $e^{-\beta}$. We further analyze Wilson loop expectations and derive an estimate that shows how the lattice scale geometry of a loop is reflected in the large $\beta$ asymptotic expansion. Specializing to axis parallel rectangular loops $\gamma_{T,R}$ with side-lengths $T$ and $R$, we consider the limiting function
    \[
        V_\beta(R) \coloneqq \lim_{T \to \infty} - \frac{1}{T} \log \, \langle W_{\gamma_{T,R}} \rangle_\beta, 
    \]
    known as the static quark potential in the physics literature. We verify existence of the limit (with an estimate on the convergence rate) and compute the first few terms in the expansion in powers of $e^{-\beta}$.
    As a consequence, a strong version of the perimeter law follows. 
    We also treat $- \log \, \langle W_{\gamma_{T,R}} \rangle_\beta / (T+R)$ as $T, R$ tend to infinity simultaneously and give analogous estimates.
\end{abstract}

\section{Introduction}

Given a hypercubic lattice $\mathbb{Z}^m$ and a choice of structure group $G$, a (pure) lattice gauge theory models a random discretized connection form on a principal $G$-bundle on an underlying discretized $m$-dimensional smooth manifold. More concretely, after restricting to a finite box,
it is a Gibbs probability measure on gauge fields, i.e., $G$-valued discrete $1$-forms $\sigma$ defined on edges of the lattice. The probability measure is defined relative to the product Haar measure on $G$. The action can be taken to be of the form $S(\sigma) = -\sum_p A_p(\sigma)$, where for some choice of representation $\rho$, $A_p(\sigma) = \textrm{Re tr } \rho(\sigma_{e_1}\sigma_{e_2}\sigma_{e_3}\sigma_{e_4})$ captures the microscopic holonomy around the plaquette $p$ whose boundary consists of the edges $e_1, \ldots, e_4$. The coupling parameter $\beta$ acts as the inverse temperature. In a formal scaling limit, one recovers the Yang-Mills action while the model enjoys exact gauge symmetry on the discrete level. In contrast to the corresponding continuum Yang-Mills theories, the discrete measure is immediately rigorously defined, and its analysis becomes a problem of statistical mechanics. 
Lattice gauge theories were introduced independently by Wegner and Wilson in the 1970s \cite{w1971,w1974}.

Despite the presence of local symmetries, lattice gauge theories can exhibit non-trivial phase structure, but one has to consider non-local observables. Given a nearest-neighbor lattice loop $\gamma$, the Wilson loop variable $W_\gamma$ records the random holonomy of the gauge field as $\gamma$ is traversed once. Starting with the original paper of Wilson \cite{w1974}, it has been argued in the physics literature that the decay rate of its expectation $\langle W_\gamma \rangle_\beta$ (in an infinite volume limit) as the loop grows encodes information about whether or not ``static quarks'' are ``confined'' in the model, see, e.g.,  of \cite[Sect.~3.5]{montway-munster} for a textbook discussion. Let $\gamma_{T,R}$ be a rectangular loop with axis parallel sides and, taking its existence for granted, consider the limit
\[
V_\beta(R) = -\lim_{T \to \infty}\frac{1}{T}\log \, \bigl|\langle W_{\gamma_{T,R}} \rangle_\beta \bigr|.
\]
The function $V_\beta(R)$ is called the static quark potential and is interpreted as the energy required to separate a static quark-antiquark pair to distance $R$, see, e.g., \cite[Sect.~3.2]{montway-munster}. Wilson's criterion for quark confinement can then be formulated as follows: confinement occurs at $\beta$ if and only if the energy $V_\beta(R)$ diverges as $R \to \infty$. However, except in the special case of planar theories, it seems that detailed mathematical proofs of such statements are not available in the literature, even for finite abelian $G$. Instead, the two phases are rigorously separated via estimates: confinement occurs at $\beta$ if there exists some function $V(R)$, unbounded as $R \to \infty$, such that  $\liminf_{T \to \infty} -\frac{1}{T}  \log \, | \langle W_{\gamma_{R,T}} \rangle_\beta | \ge  V(R)$, and in this case, Wilson loop expectations are said to follow the area law. If, on the other hand, there is a constant $c>0$ independent of $R$ such that $\limsup_{T \to \infty} - \frac{1}{T} \log\, | \langle W_{\gamma_{R,T}} \rangle_\beta |  < c$, the Wilson loop expectations are said to follow the perimeter law. (The terminology comes from the expectation that a priori bounds of the form  $e^{-c RT } \lesssim |\langle W_{\gamma_{R, T}} \rangle_\beta| \lesssim e^{-C (R+T)}$ should be essentially saturated in the two phases.) See \cite{c2021} for a precise formulation of a condition for confinement and a general discussion from a probabilistic perspective, and Section~\ref{sec: other work} below for a brief discussion of other related work.


Here we will consider lattice gauge theory with structure group $G=\mathbb{Z}_2$ on $\mathbb{Z}^m, m \ge 3$, also known as Ising lattice gauge theory, for  $\beta$ in the subcritical regime. See Section~\ref{sec: preliminaries1} for the precise definition. This model was first studied by Wegner~\cite{w1971} and can be viewed as a version of the usual Ising model on $\mathbb{Z}^m$ where the global spin flip symmetry has been ``upgraded'' to a local symmetry. We employ a cluster expansion to study the free energy, static quark potential, and related quantities. This classical method has been used in the past to analyze lattice gauge theories; see Section~\ref{sec: other work}. While we only work with $G = \mathbb{Z}_2$, we believe our results can be generalized to any choice of finite structure group with minor modifications. 

In order to state our main results we need to give some definitions.


\subsection{Ising lattice gauge theory and Wilson loop expectations}\label{sec: preliminaries1}
Let \( m \geq 3 \). The lattice $\mathbb{Z}^m$ has a vertex at each point \( x \in \mathbb{Z}^m \) with integer coordinates and a non-oriented edge between each pair of nearest neighbors. To each non-oriented edge \( \bar e \) in \( \mathbb{Z}^m \) we associate two oriented edges \( e_1 \) and \( e_2 = -e_1 \) with the same endpoints as \( \bar e \) and opposite orientations. 

Let \( \mathbf{e}_1 \coloneqq (1,0,0,\ldots,0)\), \( \mathbf{e}_2 \coloneqq (0,1,0,\ldots, 0) \), \ldots, \( \mathbf{e}_m \coloneqq (0,\ldots,0,1) \) be oriented edges corresponding to the unit vectors in \( \mathbb{Z}^m \).

If \( v \in \mathbb{Z}^m \) and \( j_1 <   j_2 \), then \( p = (v +  \mathbf{e}_{j_1}) \land  (v+ \mathbf{e}_{j_2}) \) is a positively oriented 2-cell, also known as a positively oriented plaquette. 
We let \( B_N \) denote the set \(   [-N,N]^m \) of \( \mathbb{Z}^m \), and we let \( V_N \), \( E_N \), and \( P_N \) denote the sets of oriented vertices, edges, and plaquettes, respectively, whose end-points are all in \( B_N \).

We let \( \Omega^1(B_N,\mathbb{Z}_2) \) denote the set of all  \( \mathbb{Z}_2 \)-valued  1-forms \( \sigma \) on \( E_N \), i.e., the set of all \( \mathbb{Z}_2 \)-valued functions \(\sigma \colon  e \mapsto \sigma_e \) on \( E_N \) such that \( \sigma_e =  -\sigma_{-e} \) for all \( e \in E_N \). We write $\rho: \mathbb{Z}_2 \to \mathbb{C},\, g \mapsto e^{\pi i g}$ for the natural representation of $\mathbb{Z}_2$.

When \( \sigma \in \Omega^1(B_N,\mathbb{Z}_2) \) and \( p \in P_N \), we let \( \partial p \) denote the four edges in the oriented boundary of \( p \) and define
\begin{equation*}
    (d\sigma)_p \coloneqq \sum_{e \in \partial p} \sigma_e.
\end{equation*} 
Elements \( \sigma \in \Omega^1(B_N,\mathbb{Z}_2)  \) are referred to as \emph{gauge field configurations}.

The Wilson action functional for pure gauge theory is defined by (see, e.g.,~\cite{w1974})
\begin{equation*}
    S(\sigma) \coloneqq - \sum_{p \in P_N}  \rho \bigl( (d\sigma)_p\bigr), \quad \sigma \in \Omega^1(B_N,\mathbb{Z}_2) .
\end{equation*}
The Ising lattice gauge theory probability measure on gauge field configurations is defined by
\[
\mu_{\beta,N}(\sigma)  \coloneqq
    Z_{\beta,N}^{-1} e^{-\beta S(\sigma)} , \quad \sigma \in \Omega^1(B_N,\mathbb{Z}_2) .
\]
Here for $N \in \mathbb{N}$, \[Z_{\beta,N} = \sum_{\sigma \in  \Omega^1(B_N,\mathbb{Z}_2) } e^{-\beta S(\sigma)}\] is the partition function and while we only consider the probability measure for positive $\beta$, the partition function is defined for $\beta \in \mathbb{C}$ when $N < \infty$. For $\beta \ge  0$, the corresponding expectation is written $\mathbb{E}_{\beta,N}.$ Let $\gamma$ be a nearest neighbor loop on $\mathbb{Z}^m$ contained in $B_N$. Given $\sigma \in \Omega^1(B_N,\mathbb{Z}_2) $, the Wilson loop variable for Ising lattice gauge theory is defined by
\[
W_\gamma = \rho\bigl( \sigma(\gamma) \bigr) = \prod_{e \in \gamma} \rho \bigl(\sigma(e) \bigr) = e^{\pi i \sum_{e \in \gamma} \sigma(e)} .
\]
For \( \beta \geq 0 \), let \( \langle W_\gamma \rangle_{\beta} \) denote the infinite volume limit of its expected value:
\[
    \langle W_\gamma \rangle_{\beta} \coloneqq \lim_{N \to \infty} \mathbb{E}_{\beta,N}[W_\gamma].
\]
See, e.g., \cite{flv2020} for a proof of the existence of this limit.
\subsection{Main results}

Our first result concerns the free energy for free boundary conditions. Define for $m \ge 3$ 
\begin{equation}\label{eq: beta0}
        \beta_0 = \beta_0(m) \coloneqq  
        2^{-1} \log 10(m-2) + 6^{-1} 
\end{equation} 
and let \( |P_N^+|\) be the number of positively oriented plaquettes in the restriction of \( \mathbb{Z}^m\) to the set \( [-N,N]^m\). Note that $|P_N^+| \sim \binom{m}{2}(2N)^m$ as $N\to \infty$. 
\begin{theorem}[Free energy]\label{theorem: logZ}
    Suppose $m \ge 3$ and  $\real \beta > \beta_0(m)$. Then 
    \begin{equation*}
        \begin{split}
           F(\beta)\coloneqq \lim_{N \to \infty} \frac{1}{|P_N^+|}\log Z_{\beta,N},
        \end{split} 
    \end{equation*}
  defines an analytic function, and
    \[
    F(\beta)=\frac{2}{m-1}e^{-8(m-1)\beta}
            +
            \frac{12(m-1) -8 }{2(m-1)-1}e^{-4(4(m-1)-2)\beta}
            +
            O(e^{-16(m-1)\real \beta}).
    \]
\end{theorem}
We next consider Wilson loop expectations. Given a loop $\gamma$ let \( \ell \coloneqq |\support \gamma|\) be its length, i.e., the number of edges of $\gamma$. Further, let \( \ell_c \coloneqq \ell_c(\gamma)\) denote the number of pairs of non-parallel edges that are both in the boundary of some common plaquette (corners), and let \( \ell_{b}\coloneqq\ell_b(\gamma)\) denote the number of pairs \( (e,e')\) of parallel edges that are both in the boundary of some common plaquette (bottlenecks). Set
\[
v_\beta \coloneqq 2 e^{-8(m-1)\beta}+12(m-1) e^{-4(4(m-1)-2)\beta}.
\]
\begin{theorem}\label{theorem: main theorem}
    Suppose $m \ge 3$ and  \(\beta > \beta_0(m)\). There exists $C < \infty$
 depending only on $m$ such that for any loop $\gamma$ with length \( \ell ,\) \( \ell_c\) corner edges, and \( \ell_b\) bottleneck edges, 
    \begin{equation} \label{eq: main theorem}
             \left|-\frac{1}{\ell}\log \, \langle W_\gamma \rangle_\beta - \left(v_\beta-4\frac{\ell_c+\ell_{b}}{\ell} e^{-4(4(m-1)-2)\beta} \right)\right|
            \leq C e^{-16(m-1)\beta}.
    \end{equation}
\end{theorem}
    Notice how the lattice scale geometry of the loop enters into the estimate \eqref{eq: main theorem}. Given a continuum loop, we see that the expansion is sensitive to the way the loop is embedded and discretized. For instance, the term \( (\ell_c + \ell_b)/\ell\) is very different for an axis-parallel square compared to the natural discretization of the same square rotated by \( 45^\circ.\) 
\begin{remark}
    Using the methods of the proof of Theorem~\ref{theorem: main theorem}, it is, in principle, straightforward to obtain estimates with higher precision in terms of the expansion in powers of $e^{-\beta}$. If higher order terms are included in~\eqref{eq: main theorem}, the constants of the corresponding polynomial in \( e^{-\beta}\) will further depend on the lattice scale geometry of the loop. 
\end{remark}
Our next result concerns the static quark potential $V_\beta(R)$. 
\begin{theorem}[Quark potential and perimeter law]\label{theorem: limit of ratio exists 2}
    Suppose $m \ge 3$ and  $\beta > \beta_0(m)$. There exists a function $V_\beta(\cdot)$ and a constant $C < \infty$ such that the following holds. Let \( R \geq 2 \) be an integer and for \( T = 1, 2, \ldots \) let \( \gamma_{R,T}\) be a rectangular loop with side lengths \( R \) and \( T\) and axis-parallel sides.  
    Then, 
    \begin{equation*}
   \left| -\frac{1}{T}\log \, \langle W_{\gamma_{R,T}} \rangle_\beta - V_\beta(R)\right| \le \frac{C}{T}.
    \end{equation*}
    The limit $V_\beta:= \lim_{R \to \infty}V_\beta(R)$ exists and
    \[
    V_\beta = 4 e^{-8(m-1)\beta} + 24 (m-1)e^{-4(4(m-1)-2)\beta} + O(e^{-16(m-1)\beta}).
    \]
\end{theorem}
By the theorem, $V_\beta(R)$ exists for all sufficiently large $\beta$ and is bounded as $R \to \infty$ so we obtain a proof of the perimeter law. Moreover, using the convergence rate estimate we also obtain the up-to-constants estimate
\[
 \langle W_{\gamma_{R,T}} \rangle_\beta \asymp e^{-T V_\beta(R)}, \quad T \to \infty.
\]
\begin{remark}
    At fixed $R$, we have the expansion as $\beta \to \infty$
    \[
    V_\beta(R) =  4 e^{-8(m-1)\beta} + 24(m-1) e^{-4(4(m-1)-2)\beta} + O_R(1)e^{-16(m-1)\beta}.
    \]
\end{remark}

\begin{remark}
    We use a cluster expansion to prove Theorem~\ref{theorem: limit of ratio exists 2}, including the existence part. Alternatively, one could prove the existence of \( V_\beta(R) \) using  Griffith's second inequality to deduce subadditivity and then appeal to Fekete's lemma. However, this method would give no quantitative information as in the theorem. Moreover, it cannot be used to obtain Propositions~\ref{eq: proposition rectangle limit 1}~and~\ref{eq: proposition rectangle limit 2}, which shows the existence of the limit in some generality and is needed for Theorem~\ref{theorem: limit of ratio exists} below.
\end{remark}

Our next result is a version of Theorem~\ref{theorem: limit of ratio exists 2} in the setting where the two sides of the loop grow uniformly. 
\begin{theorem}\label{theorem: limit of ratio exists}
    Suppose $m \ge 3$ and  $\beta > \beta_0(m)$, let \( r,t \geq 1 \) be integers and for \( n = 1, 2, \ldots \) let \( \gamma_{n}\) be an axis-parallel rectangular loop with side lengths \( R_n = rn \) and \( T_n =tn.\)  
    Then 
    \begin{equation*}
       \lim_{n\to \infty}- \frac{1}{R_n+T_n} \log \, \langle W_{\gamma_{n}} \rangle_\beta = V_\beta,
    \end{equation*}
    where $V_\beta=\lim_{R\to \infty }V_\beta(R)$.
\end{theorem}
Note that $V_\beta = 2v_\beta +  O(e^{-16(m-1)\beta})$. We have chosen to state Theorem~\ref{theorem: limit of ratio exists} for a rectangle but with small modifications the proof is also valid for any fixed loop for which the proportion of corners and bottlenecks in the scaled loop tend to zero as \( n \to \infty\) and the corresponding $V_\beta$ is the same. 

\begin{remark}
It would be interesting to relate the confinement phase transition to analyticity properties of the functions $\beta \mapsto F(\beta)$ and $\beta \mapsto V_\beta$. 
\end{remark}

\subsection{Related work and further comments}\label{sec: other work}
  We refer to \cite{c2018} for a thorough discussion of classical works on area-/perimeter law estimates in various settings, including \cite{os1978, g1980, g1980b, fs1982, gm1982, seiler}. Among more recent results, we mention \cite{gs2023}, which considers the 4-dimensional $U(1)$ theory with Villain action. In the perimeter law regime, for sufficiently regular loops, it was shown that
    \[
    \frac{C_{0}}{2\beta}(1+C \beta e^{-2\pi^2\beta})(1 + o(1)) \le - \frac{1}{|\gamma|}\log |\mathbb{E}_{\beta}^{Vil}[W_\gamma]| \le  \frac{C_{0}}{2\beta}(1+\epsilon(\beta))(1 + o(1)),
    \]
    where the upper bound was due to Fr\"ohlich and Spencer \cite{fs1982}. Here $C_0$ is a constant related to the discrete Gaussian free field. The infinite volume free energy for this model was also considered in \cite{gs2023}, and an upper bound was obtained for the ``internal energy'', i.e., its derivative with respect to $\beta$.  
    
In the important paper \cite{c2019}, Chatterjee studied 4-dimensional Ising lattice gauge theory. Using a resampling argument, it was proved that  
\begin{equation}\label{eq: easy upper bound}
    \langle W_\gamma \rangle_\beta \leq e^{-\frac{1-\ell_c/\ell}{1 + e^{-4(m-1)\beta}} \, 2\ell e^{-4(m-1)\beta}}.
\end{equation} 
(We caution that \( 2\beta\) in the present paper is equal to the parameter \( \beta\) used in \cite{c2019}.) This estimate is valid for all \( \beta > 0\).
Using~\eqref{eq: easy upper bound}, the inequalities
\begin{equation}\label{eq: ci1}
    \bigl| \langle W_\gamma \rangle_\beta - e^{-2\ell e^{-8(m-1) \beta}} \bigr| \leq
    Ce^{\frac{- 2}{1 + e^{-16(m-1)\beta}} \,\ell e^{-8(m-1)\beta}}\bigl( e^{-8\beta} + \sqrt{\ell_c/\ell} \bigr)
\end{equation}
and
\begin{equation*}
    \bigl| \langle W_\gamma \rangle_\beta - e^{-2\ell e^{-8(m-1) \beta}} \bigr| \leq C_1 \bigl( e^{-8\beta} + \sqrt{\ell_c/\ell} \bigr)^{C_2}
\end{equation*}
were obtained. The ideas introduced in \cite{c2019} spurred several recent works and analogous estimates have now been given in more general settings, including for arbitrary finite structure groups and for corresponding lattice Higgs models, see~\cite{a2021, sc2019, flv2021, f2021b, flv2020}.
The methods used in these papers produce error terms that will generally be larger than the estimate for \( \langle W_\gamma \rangle \) if one does not have a relation of the type \( \ell e^{-2\ell e^{-8(m-1)\beta}} \ll \infty.\) That is, one needs the size of the loop to tend to infinity at a rate tuned to $\beta \to \infty$. (Of course, one sees different exponents for different choices of structure group $G$; this case corresponds to Ising lattice gauge theory.) As a consequence, it is not clear (to us) how to use those methods to prove a perimeter law (lower bound) estimate at fixed $\beta$ or, e.g., how to analyze limits such as the one defining the quark potential. Moreover, we do not know how they can easily be modified to obtain higher precision even if the loop grows with $\beta$ at an appropriate rate. 

Here we instead carry out the analysis based on the cluster expansion of the partition function, which provides information on $ \log \, \langle W_\gamma \rangle_\beta$. One still needs $\beta$ to be sufficiently large, but it does not need to grow with \( \ell \) for the error bounds to be small, and the drawbacks discussed above can be circumvented. The method yields, in principle, arbitrary precision for the logarithm of the Wilson loop expectation and also allows to quantify the behavior of \( \langle W_\gamma \rangle \) when \( \ell e^{-8(m-1)\beta} \ll \infty.\) This partly resolves one of the open problems in~\cite{c2019}. However, the work here does not directly imply the results of~\cite{c2019} but do give alternative proofs of several of the key lemmas therein. 
 
The use of the cluster expansions in the context of lattice gauge theories context is certainly not new, see in particular Seiler's monograph \cite{seiler} (and the references therein) where, e.g., perimeter law estimates to first order for large $\beta$ were obtained. However, besides basic facts about the cluster expansion as presented in the recent textbook of Friedly and Velenik \cite{fv2017} and some results from~\cite{flv2020, f2021b}, our discussion is self-contained, and we carry out all the needed estimates here.

We only consider $G = \mathbb{Z}_2$ in this paper. We expect that one can extend the results to any finite group $G = \mathbb{Z}_k, k \ge 3,$ without much additional effort, as well as to the corresponding lattice Higgs models. It also seems plausible that, with more work, any finite $G$ can be analyzed similarly. The cluster expansion based on vortices uses crucially that gauge field configurations can be split into discrete components. Therefore we do not expect the methods in this paper to work in the general case of compact subgroups of $U(N)$.

\subsection{Acknowledgements} M.P.F. acknowledges support from the Ruth and Nils-Erik Stenb\"ack Foundation.  F.V. acknowledges support from the Knut and Alice Wallenberg Foundation and the Swedish Research Council. We thank Juhan Aru and Christophe Garban for interesting discussions.

\section{Preliminaries}
Even though we later work with $G = \mathbb{Z}_2$, in this section we allow $G$ to be a general finite abelian group since this entails no additional work.
\subsection{Notation and standing assumptions}
In the rest of this paper, we assume that \( m \geq 3 \) is given. We define the dimension dependent constant
\[
M = M(m) \coloneqq10(m-2).
\]
We note that with this notation, we have 
\begin{equation*} 
        \beta_0 = \beta_0(m) = 2^{-1} \log M + 6^{-1}   .
\end{equation*}

\subsection{Discrete exterior calculus}
In order to keep the presentation short, and since these definitions have appeared in several recent papers, we will refer to \cite{flv2020} for details on some of the basic notions of discrete exterior calculus that is useful in the present context. 

\begin{itemize}
\item  We will work with the square lattice $\mathbb{Z}^m$, where we assume that the dimension $m \ge 3$ throughout. We write $B_N = [-N,N]^m \cap \mathbb{Z}^m$.
\item{We write $C_k(B_N)$ and $C_k(B_N)^+$ for the set of unoriented and positively oriented $k$-cells, respectively (see \cite[Sect. 2.1.2]{flv2020}). Note that in the introduction, we used \( V_N = C_0(B_N),\) \( E_N = C_1(B_N) ,\) and \( P_N = C_2(B_N) . \) An oriented $2$-cell is called a plaquette.}
\item{Formal sums of positively oriented \( k \)-cells with integer coefficients are called $k$-chains, and the space of $k$-chains is denoted by \( C_k(B_N,\mathbb{Z}) \), (see \cite[Sect. 2.1.2]{flv2020})}
\item{Let \( k \geq 2 \) and \( c = \frac{\partial}{\partial x^{j_1}}\big|_a \wedge \dots \wedge \frac{\partial}{\partial x^{j_k}}\big|_a \in C_k(B_N)\). The \emph{boundary} of $c$ is the $(k-1)$-chain \(\partial c \in C_{k-1}(B_N, \mathbb{Z})\) defined as the formal sum of the \( (k-1)\)-cells in the (oriented) boundary of \( c.\) 
The definition is extended to $k$-chains by linearity. See \cite[Sect.~2.1.4]{flv2020}..}
\item{If \( k \in \{ 0,1, \ldots, n-1 \} \) and \( c \in C_k(B_N)\) is an oriented \( k \)-cell, we define the \emph{coboundary} \( \hat \partial c \in C_{k+1}(B_N)\) of \( c \) as the \( (k+1) \)-chain $\hat \partial c \coloneqq \sum_{c' \in C_{k+1}(B_N)} \bigl(\partial c'[c] \bigr) c'.$ See \cite[Sect.~2.1.5]{flv2020}.}

\item{We let $\Omega^k(B_N, G)$ denote the set of $G$-valued (discrete differential) $k$-forms (see \cite[Sect 2.3.1]{flv2020}); the exterior derivative $d : \Omega^k(B_N, G)  \to \Omega^{k+1}(B_N, G)$ is defined for $0 \le k \le m-1$ (see \cite[Sect. 2.3.2]{flv2020}) and $\Omega^k_0(B_N, G)$ denotes the set of closed $k$-forms, i.e., $\omega \in \Omega^k(B_N, G)$ such that $d\omega = 0$.}
\item{We write $\support \omega = \{c \in C_k(B_N): \omega(c) \neq 0\}$ for the support of a $k$-form $\omega$. Similarly, we write $(\support \omega)^+ = \{c \in C_k(B_N)^+: \omega(c) \neq 0\}$}
\item{
    A 1-chain \( \gamma \in C_1(B_N,\mathbb{Z}) \) with finite support \( \support \gamma \) is called a \emph{loop} if 
    for all \( e \in \Omega^1(B_N) \), we have that \( \gamma[e] \in \{ -1,0,1 \} \), and \( \partial \gamma = 0. \) We write \( |\gamma| = |\support \gamma|.\) (In \cite{flv2020} this object was called a generalized loop.)
}

\item  Let \( \gamma \in C_1(B_N,\mathbb{Z}) \) be a loop. A \( 2 \)-chain \( q \in C_2(B_N,\mathbb{Z}) \) is an \emph{oriented surface} with \emph{boundary} \( \gamma \) if \( \partial q = \gamma. \) 
\end{itemize}

\subsection{Plaquette adjacency graph}\label{sec: the graphs}

Let \( \mathcal{G}_2 \) be the graph with vertex set \( C_2(B_N)^+\) and an edge between two distinct vertices \( p_1,p_2 \in C_2(B_N)^+\) if and only if \( \support \hat \partial p_1 \cap \support \hat \partial p_2 \neq \emptyset.\) 

Since any plaquette \( p \in C_2(B_N)^+\) in \( B_N\) is in the boundary of at most \( 2(m-2)\) 3-cells, and any such 3-cell has exactly five plaquettes in its boundary that are not equal to \( p,\) it follows that there are at most \(  5 \cdot 2(m-2) = 10(m-2) = M \) plaquettes \( p' \in C_2(B_N)^+\smallsetminus \{ p \} \) with \( {\support \hat \partial p \cap \support \hat \partial p' \neq \emptyset.}\) Therefore it follows that each vertex in \( \mathcal{G}_2 \) has degree at most \( M.\)

\subsection{Vortices}

\begin{definition}[Vortex]
    A closed 2-form \( \nu \in \Omega^2_0(B_N,G)\) is said to be a \emph{vortex} if \( (\support \nu)^+\) induces a connected subgraph of \( \mathcal{G}_2\). 
\end{definition}

The set of all vortices in \( \Omega^2(B_N,G)\) is denoted by \( \Lambda.\) 
We note that the definition of vortex we use here is not exactly the same as the definition used in~\cite{f2021,f2021b,flv2020,flv2021}, but agrees with the definition used in~\cite{sc2019, c2019}.

When \( \omega,\nu \in \Omega^2(B_N,G),\) we say that \( \nu \) is a vortex in \( \omega \) if  \(  \nu\) is a vortex and \( \support \nu \) corresponds to a connected subgraph of the subgraph of \( \mathcal{G}_2\) induced by \( \support \omega. \) 

\begin{lemma}[The Poincar\'e lemma, Lemma 2.2 in~\cite{c2019}]\label{lemma: poincare}
    Let \( k \in \{ 1, \ldots, m\} \) and let \( B \) be a box in \( \mathbb{Z}^m \). Then the exterior derivative \( d \) is a surjective map from the set \( \Omega^{k-1}(B \cap \mathbb{Z}^m, G) \) to \( \Omega^k_0(B \cap \mathbb{Z}^m, G) \).
    Moreover, if \(G\) is finite, then this map is an \( \bigl| \Omega^{k-1}_0(B \cap \mathbb{Z}^m, G)\bigr|\)-to-\(1\) correspondence.
    Lastly, if \( k \in \{ 1,2, \ldots, m-1 \} \) and \(\omega \in \Omega^k_0(B \cap \mathbb{Z}^m, G)\) vanishes on the boundary of \(B\), then there is a \((k-1)\)-form \( \omega' \in \Omega^{k-1}(B \cap \mathbb{Z}^m, G)\) that also vanishes on the boundary of \(B\) and satisfies \(d\omega' = \omega\). 
\end{lemma}

\begin{lemma}[Lemma~2.4 of~\cite{f2021}]\label{lemma: small vortices}
    Let \( \omega \in \Omega_0^2(B_N,G). \)  If \( \omega \neq 0 \) and the support of \( \omega \) does not contain any boundary plaquettes of \( B_N, \) then either \( \bigl| (\support \omega)^+ \bigr| = 2(m-1) \), or \( \bigl|(\support \omega)^+| \geq 4(m-1)-2. \)
\end{lemma}
In~\cite{f2021}, we proved Lemma~\ref{lemma: small vortices} only in the case \( m = 4,\) but since the proof for general \( m \geq 2\) is analogous we do not include it here.

In Figure~\ref{table: minimal configurations}, we illustrate the only two possibilities for \( (\support \omega)^+\) if \( \bigl| (\support \omega)^+ \bigr| = 4(m-1)-1\) when \( m = 4.\) For general \( m \geq 2,\) the situation is analogous.
\begin{figure}[ht]
    \centering
    \begin{tabular}{m{1cm} p{2.5cm} p{4cm} m{2cm}}
         & \( (\support \sigma)^+ \) & \(  (\support d\sigma)^+ \) &  
         \( \bigl| (\support \omega)^+ \bigr| \)
        \\[0.5ex] \toprule \\[-1ex]
        \hfil (a) \hfil \vspace{7.5ex} & \begin{tikzpicture}
            \draw[white] (-0.5,1.3) -- (0.5,-0.3);
            \draw[thick, detailcolor00] (0,0) -- (0,1);
        \end{tikzpicture} & 
        \begin{tikzpicture}
            \filldraw[fill=detailcolor07, fill opacity=0.18] (0,0) -- (1,0) -- (1,1) -- (0,1) -- (0,0);
            \filldraw[fill=detailcolor07, fill opacity=0.18] (0,0) -- (0,1) -- (-1,1) -- (-1,0) -- (0,0);
            \filldraw[fill=detailcolor07, fill opacity=0.18] (0,0) -- (0.55,-0.3) -- (0.55,0.7) -- (0,1) -- (0,0);
            \filldraw[fill=detailcolor07, fill opacity=0.18] (0,0) -- (0,1) -- (-0.55,1.3) -- (-0.55,0.3) -- (0,0);
            \filldraw[fill=detailcolor07, fill opacity=0.18] (0,0) -- (0.7,0.25) -- (0.7,1.25) -- (0,1) -- (0,0);
            \filldraw[fill=detailcolor07, fill opacity=0.18] (0,0) -- (0,1) -- (-0.7,0.75) -- (-0.7,-0.25) -- (0,0);
            \draw[thick, detailcolor00] (0,0) -- (0,1);
        \end{tikzpicture} &  
        \hfil \(2(m-1)\) \hfil \vspace{8.5ex}
        \\  
        \hfil (b) \hfil \vspace{7.5ex} & 
        \begin{tikzpicture}
            \draw[white] (0,1.3) -- (0,-0.3);
            
            \draw[thick, detailcolor00] (0,0) -- (0,1);
            \draw[thick, detailcolor00] (1,0) -- (1,1);
        \end{tikzpicture} & 
        \begin{tikzpicture}  
            
            \filldraw[fill=detailcolor07, fill opacity=0.18] (0,0) -- (0,1) -- (-1,1) -- (-1,0) -- (0,0);
            \filldraw[fill=detailcolor07, fill opacity=0.18] (0,0) -- (0.55,-0.3) -- (0.55,0.7) -- (0,1) -- (0,0);
            \filldraw[fill=detailcolor07, fill opacity=0.18] (0,0) -- (0,1) -- (-0.55,1.3) -- (-0.55,0.3) -- (0,0);
            \filldraw[fill=detailcolor07, fill opacity=0.18] (0,0) -- (0.7,0.25) -- (0.7,1.25) -- (0,1) -- (0,0);
            \filldraw[fill=detailcolor07, fill opacity=0.18] (0,0) -- (0,1) -- (-0.7,0.75) -- (-0.7,-0.25) -- (0,0);
            
            \filldraw[fill=detailcolor07, fill opacity=0.18] (1,0) -- (2,0) -- (2,1) -- (1,1) -- (1,0);
            \filldraw[fill=detailcolor07, fill opacity=0.18] (1,0) -- (1.55,-0.3) -- (1.55,0.7) -- (1,1) -- (1,0);
            \filldraw[fill=detailcolor07, fill opacity=0.18] (1,0) -- (1,1) -- (0.45,1.3) -- (0.45,0.3) -- (1,0);
            \filldraw[fill=detailcolor07, fill opacity=0.18] (1,0) -- (1.7,0.25) -- (1.7,1.25) -- (1,1) -- (1,0);
            \filldraw[fill=detailcolor07, fill opacity=0.18] (1,0) -- (1,1) -- (0.3,0.75) -- (0.3,-0.25) -- (1,0);
            
            \draw[thick, detailcolor00] (0,0) -- (0,1);
            \draw[thick, detailcolor00] (1,0) -- (1,1);
            
        \end{tikzpicture} &  
        \hfil \(4(m-1)-2\) \hfil \vspace{5.5ex}
        \\  
        \hfil (c) \hfil \vspace{14.5ex} & 
        \begin{tikzpicture}
            \draw[white] (0.,1.8) -- (0.,-0.9);
            
            \draw[thick, detailcolor00] (1,0) -- (0,0) -- (0,1);
        \end{tikzpicture} &  
        \begin{tikzpicture}
         
            \draw[white] (0,1.3) -- (0,-1.3);  
            
            \filldraw[fill=detailcolor07, fill opacity=0.18] (0,0) -- (0,1) -- (-1,1) -- (-1,0) -- (0,0);
            \filldraw[fill=detailcolor07, fill opacity=0.18] (0,0) -- (0.55,-0.3) -- (0.55,0.7) -- (0,1) -- (0,0);
            \filldraw[fill=detailcolor07, fill opacity=0.18] (0,0) -- (0,1) -- (-0.55,1.3) -- (-0.55,0.3) -- (0,0);
            \filldraw[fill=detailcolor07, fill opacity=0.18] (0,0) -- (0.7,0.25) -- (0.7,1.25) -- (0,1) -- (0,0);
            \filldraw[fill=detailcolor07, fill opacity=0.18] (0,0) -- (0,1) -- (-0.7,0.75) -- (-0.7,-0.25) -- (0,0); 
            
            \filldraw[fill=detailcolor07, fill opacity=0.18] (0,0) -- (1,0) -- (1,-1) -- (0,-1) -- (0,0);
            \filldraw[fill=detailcolor07, fill opacity=0.18] (0,0) -- (0.55,-0.3) -- (1.55,-0.3) -- (1,0) -- (0,0);
            \filldraw[fill=detailcolor07, fill opacity=0.18] (0,0) -- (-0.55,0.3) -- (0.45,0.3) -- (1,0) -- (0,0);
            \filldraw[fill=detailcolor07, fill opacity=0.18] (0,0) -- (-0.7,-0.25) -- (0.3,-0.25) -- (1,0) -- (0,0);
            \filldraw[fill=detailcolor07, fill opacity=0.18] (0,0) -- (0.7,0.25) -- (1.7,0.25) -- (1,0) -- (0,0);
            
            \draw[thick, detailcolor00] (1,0) -- (0,0) -- (0,1);
            
        \end{tikzpicture} &
        \hfil  \(4(m-1)-2\) \hfil \vspace{14.5ex}
\end{tabular} 
\vspace{-8ex}
\caption{The above table shows projections of the supports of the  non-trivial and irreducible plaquette configurations in \( \mathbb{Z}^4 \) which has the smallest support (up to translations and rotations).%
}\label{table: minimal configurations}
\end{figure}

\begin{lemma}[Lemma~4.6 in~\cite{flv2020}]\label{lemma: minimal vortices}
    Let \( \omega \in \Omega_0^2(B_N,G) \).  If the support of \( \omega \) does not contain any boundary plaquettes of \( B_N \) and \( \bigl| (\support \omega)^+ \bigr| = 2(m-1)\), then there is an edge \( e \in C_1(B_N) \) and \( g \in G\smallsetminus \{ 0 \} \) such that 
        \begin{equation}\label{eq: minimal vortex}
            \omega = d\bigl(g \, e \bigr). 
        \end{equation} 
\end{lemma}

If \( \omega \in \Omega^2(B_N,G) \) is such that~\eqref{eq: minimal vortex} holds for some \( e \in C_1(B_N)\) and \( g \in G \smallsetminus \{ 0 \},\) then we say that \( \omega\) is a \emph{minimal vortex around \( e .\)}

\begin{lemma}\label{lemma: minimal 2}
    Let \(\omega \in \Omega_0^2(B_N,G),\) and assume that the support of \( \omega \) does not contain any boundary plaquettes of \( B_N \) and \(\bigl|(\support \omega)^+ \bigr| = 4(m-1)-2.\) Then there are two distinct edges \( e,e' \in C_1(B_N)\) with \( (\hat \partial e)^+ \cap (\hat \partial e')^+ \neq \emptyset\) and \( \sigma \in \Omega^1(B_N,G)\) with \( (\support \sigma)^+  =\{ e,e' \} \) such that \( d \sigma = \omega. \)
\end{lemma}

For a proof of Lemma~\ref{lemma: minimal 2}, see the proof of~\cite[Lemma~2.4]{f2021}.

    



 
    

\begin{lemma}\label{lemma: polymer and box}
    Let \( q\) be an oriented surface with \( \partial q = \gamma.\) Further, let \( \omega \in \Omega^2(B_N,G)\) be such that \( d\omega = 0 \) and \( \omega(q) \neq 0.\) Then any box which contains \( \support \omega\) must intersect an edge in \( \support \gamma.\) 
\end{lemma}

\begin{proof}
    Let \( B\) be a box that contains \( \support \omega.\) Since  \( d\omega = 0,\)  by the Poincaré lemma (see e.g.~\cite[Lemma 2.2]{flv2020} there is \( \sigma \in \Omega^1(B_N,G)\) whose support is contained in \( B \) such that \( d \sigma = \omega.\) Moreover, we have \( \omega(q) = \sigma(\gamma)\) (see e.g. \cite[Section 2.4]{flv2020}. Consequently, if \( B \) does not intersect \( \support \gamma, \) then \( \omega(q) = \sigma(\gamma) = 0\), a contradiction. 
\end{proof}

\begin{lemma}\label{lemma: plaquette at boundary}
        Let \( \nu \in \Omega^2(B_N,G)\) satisfy \( d\nu = 0,\) let \( B \) be a box that contains the support of \( \nu \) and let \( p \in \support \nu.\) Then there is at least one 1-cell in \( \support \partial p\) that is not in the boundary of \( B.\)
    \end{lemma}

    \begin{proof}
        Assume for contradiction that all edges in \( \support \partial p\) are in the boundary of \( B.\) Then there is a 3-cell \( c \in \hat \partial p\) that is not contained in \( B.\) Since \( B\) is a box, \( p \) must be the only  plaquette in \( \support \partial c\) that is in \( B.\) Since the support of \( \nu\) is contained in \( B,\) it follows that
        \[
        d\nu(c) = \sum_{p' \in \partial c} \nu(p') = \nu(p) \neq 0.
        \]
        Since this contradicts the assumption that \( d\nu=0,\) the desired conclusion follows.
    \end{proof}

The following lemma is elementary.
\begin{lemma}\label{lemma: plaquettes on a given distance}
    There is a constant \( C_m>0 \) such that for any \( e \in C_1(B_N)^+ \) and   \( j \geq 0, \) we have
    \begin{equation*}
        \bigl| \{ p \in C_2(B_N)^+ \colon \dist(p,e) = j \} \bigr| \leq C_m \max(1,j)^{m-1}.
    \end{equation*}
\end{lemma}

\begin{lemma}\label{lemma: box}
    Let \( j \geq 1,\) let \( p \in C_0(B_N)\) and let \( B\) be a box with side lengths \( s_1,s_2,\dots,s_m \) that contains \( p \) and is such that every face of the box contains at least one point on distance at least \( j\) from \( p .\) Then \(   \sum_{i=1}^m s_i   \geq jm/(m-1) .\)
\end{lemma}

\begin{proof}
    Without loss of generality we can assume that the box \( B \) has corners at \( (0,0,\dots 0)\) and \( (s_1,s_2,\dots,s_m),\) that \( p=(x_1,x_2,\dots,x_m),\) and that \( 0 \leq x_i \leq s_i/2\) for \( i = 1,2, \dots, m. \) Then the assumption on \( B\) is equivalent to that
    \[
    x_i + \sum_{k \neq i} (s_k-x_k) \geq j
    \]
    for each \( i \in \{  1,2, \dots, m \}. \) Summing over \( i,\) we obtain
    \begin{align*}
        &\sum_{i = 1}^m \Bigl( x_i + \sum_{k \neq i} (s_k-x_k)\Bigr) \geq mj
        \Leftrightarrow
        \sum_{i = 1}^m x_i + (m-1) \sum_{i=1}^m s_i - (m-1) \sum_{i = 1}^m x_i  \geq mj
        \\&\qquad \Leftrightarrow
        (m-1) \sum_{i=1}^m s_i   \geq mj + (m-2) \sum_{i = 1}^m x_i.
    \end{align*}
    From this the desired conclusion immediately follows.
\end{proof}

\begin{lemma}\label{lemma: minimal size if at distance}
    There is a constant \( \hat C_m >0 \) such that for any oriented surface \( q ,\) any let \( j \geq 1, \) and any \( \nu \in \Lambda \) with \( \nu(q)=1 \) and  \( \dist(\support \nu ,\gamma)=j, \) we have \( |(\support \nu )^+| \geq  \hat C_m(j+1). \) 
\end{lemma}
 
\begin{proof}
    Let \( B\) be the (unique) smallest box that contains the support of \( \nu,\) and assume that the side lengths of \( B\) are \( s_1,\dots,s_m.\) Since \( \nu(q) = 1, \) it follows from Lemma~\ref{lemma: polymer and box} that \( B\) intersects an edge of \(\support \gamma.\) Consequently, there is some edge in \( \gamma\) whose both endpoints are contained in \( B.\) Fix one such edge \( e.\) Note that, by assumption, we have \( \dist (\support \nu,e) \geq j.\)
    Since \( B\) is a minimal box containing \( \support \nu,\) there must be one edge on each face of the box which is contained in the boundary of some plaquette in \( \support \nu.\) 
    At the same time, by Lemma~\ref{lemma: plaquette at boundary}, since \( d\nu = 0,\) no plaquette in \( \support \nu \) can be in the boundary of \( B .\) This implies in particular that each plaquette in \( p \in (\support \nu)^+\) must have an edge in its boundary that is not in the boundary of \( B.\) 
    Since for each such plaquette we must have \( \dist (p,e)\geq j,\) it follows from Lemma~\ref{lemma: box} that 
    \[
    \sum_{i=1}^m (s_i-2) \geq \frac{(j+1)m}{m-1} \Leftrightarrow \sum_{i=1}^m s_i  \geq \frac{(j+1)m}{m-1} + 2m.
    \]
    Since \( \omega \in \Lambda,\) the set \( (\support \omega)^+\) induces a connected subgraph of \( G_2.\) Since each face of \( B\) contains at least one edge that is in the boundary of some plaquette in \( (\support \omega)^+,\) the desired conclusion immediately follows.
\end{proof}

\subsection{Vortex clusters}
Recall that $\Lambda$ denotes the set of vortices in $\Omega^2_0(B_N,G)$. 

If $\nu_1,\nu_2 \in \Lambda$, we write $\nu_1 \sim \nu_2$ if there is \( p_1\in (\support \nu_1)^+\) and \( p_2 \in (\support \nu_2)^+\) such that \( p_1 \sim p_2\) in \( \mathcal{G}_2.\)

Consider a multiset 
\begin{align*}
    &\mathcal{V} = \{ \underbrace{\nu_1,\dots, \nu_1}_{n_{\mathcal{V}}(\nu_1) \text{ times}}, \underbrace{\nu_2, \dots, \nu_2}_{n_{\mathcal{V}}(\nu_2) \text{ times}},\dots, \underbrace{\nu_k,\dots,\nu_k}_{n_{\mathcal{V}}(\nu_k) \text{ times}} \} = \{\nu_1^{n(\nu_1)}, \ldots, \nu_k^{n(\nu_k)}\}, 
\end{align*}
where \( \nu_1,\dots,\nu_k \in \Lambda \) are distinct and $n(\nu)=n_{\mathcal{V}}(\nu)$ denotes the number of times $\nu$ occurs in~$\mathcal{V}$. Following \cite[Chapter 3]{fv2017}, we say that $\mathcal{V}$ is \emph{decomposable} if there exist non-empty and disjoint multisets $\mathcal{V}_1,\mathcal{V}_2 \subset \mathcal{V}$ such that $\mathcal{V} = \mathcal{V}_1 \cup \mathcal{V}_2$ and such that for each pair $(\nu_1,\nu_2) \in \mathcal{V}_1 \times \mathcal{V}_2$, we have \( \nu_1 \nsim \nu_2.\)
%
%
If $\mathcal{V}$ is not decomposable, it is by definition a \emph{vortex cluster}. We stress that a vortex cluster is unordered and may contain several copies of the same vortex. 

Given a vortex cluster $\mathcal{V} $, let us define
\begin{equation*}
    |\mathcal{V}| = \sum_{\nu \in \Lambda} n_{\mathcal{V}} (\nu) \bigl| (\support \nu )^+ \bigr|,\quad n(\mathcal{V}) = \sum_{\nu \in \Lambda} n_{\mathcal{V}} (\nu),\quad \text{and} \quad
    \support \mathcal{V} = \bigcup_{\nu \in \mathcal{V}} \support \nu.
\end{equation*}
For a $2-$chain \( q \in C_2(B_N,\mathbb{Z}),\) we define 
\[ \mathcal{V}(q) = \sum_{\nu \in \Lambda} n_{\mathcal{V}}(\nu) \nu(q).\] 

We write \( \Xi \) for the set of all vortex clusters of $\Lambda$.

To simplify notation, for \( p \in C_2(B_N),\) \( q \in C^2(B_N,\mathbb{Z}), \) \( i \geq 1,\) and \( j\geq 1,\) we define
\begin{equation*}
    \Xi_{i} \coloneqq \bigl\{ \mathcal{V} \in \Xi \colon 
    n(\mathcal{V}) = i \bigr\},
\end{equation*} 
\begin{equation*}
    \Xi_{i,j} \coloneqq \bigl\{ \mathcal{V} \in \Xi \colon 
    n(\mathcal{V}) = i,\, |\mathcal{V}| = j \bigr\},
\end{equation*} 
\begin{equation*}
    \Xi_{i,j,p} \coloneqq \bigl\{ \mathcal{V} \in \Xi \colon 
    n(\mathcal{V}) = i,\, |\mathcal{V}| = j, p \in \support \mathcal{V} \bigr\},
\end{equation*} 
and
\begin{equation*}
    \Xi_{i,j,q} \coloneqq \bigl\{ \mathcal{V} \in \Xi \colon 
    n(\mathcal{V}) = i,\, |\mathcal{V}| = j, \mathcal{V}(q) \neq 0 \bigr\}.
\end{equation*} 
Further, we let
\begin{equation*}
    \Xi_{i^+ } \coloneqq \bigl\{ \mathcal{V} \in \Xi \colon 
    n(\mathcal{V}) \geq i\} \quad \text{and} \quad \Xi_{i^- } \coloneqq \Xi_{1^+ } \smallsetminus \Xi_{i^+ }  
\end{equation*} 
and define \( \Xi_{i^+,j}, \) \( \Xi_{i,j^+}, \) \( \Xi_{i^+,j^+}, \) etc. analogously.
We note that the sets defined above depend on \( N\) but usually suppress this in the notation. When we want to emphasize this we write \( \Xi_i(B_N),\) \( \Xi_{i,j}(B_N), \) etc.

The following lemma gives an upper bound on the number of vortices of a gives size that contains a given plaquette \( p. \)
\begin{lemma}\label{lemma: the number of vortices}
    Let \( k \geq 1\) and let \( p \in C_2(B_N)^+.\) Then
    \begin{equation*}
        |\Xi_{1,k,p}|\leq M^{2k-1}.
    \end{equation*}
\end{lemma}

\begin{proof}
    Let $\{\nu\}\in \Xi_{1,k,p}$ and let \( \mathcal{P}\) be the set of all paths in \( \mathcal{G}_2 \) that starts at \( p \) and has length \( 2\bigl| (\support \nu)^+ \bigr|-1 = 2k-1.\) Since each vertex in \( \mathcal{G}_2\) has degree at most \(10(m-2)=M,\) we have  \( |\mathcal{P}| \leq M^{2k-1}.\) 

    For \( \{ \nu \} \in \Xi_{1,k,p},\) let \( G_\nu\) be the subgraph of \( \mathcal{G}_2 \) induced by the set \( (\support \nu)^+.\)
    Then \( G_\nu\) is connected, and hence \( G_\nu \) has a spanning path \( T_\nu \in \mathcal{P}\) of length \( 2\bigl| (\support \nu)^+ \bigr|-1\) which starts at \( p .\) 
    Since the map \( \nu \mapsto T_\nu\) is an injective map from \( \Xi_{1,k,p}\) to \( \mathcal{P}\) and \( |\mathcal{P}| \leq  M^{2k-1},\) the desired conclusion immediately follows.
\end{proof}

\subsection{The activity}\label{sec: activity}

For \( \beta \geq 0 \) and \( g \in G \) with a unitary, one-dimensional representation $\rho$, we set
\begin{equation*}
    \phi_{\beta}(g) \coloneqq  e^{\beta \real  (\rho(g)-\rho(0))} .
\end{equation*} 
Since \( \rho \) is unitary, for any \( g \in G \) we have \( \rho(g) = \overline{\rho(-g)} \), and hence \( \real  \rho(g) = \real  \rho(-g) \). In particular,  for any \( g \in G \) 
\begin{equation} \label{eq: phi is symmetric}
    \phi_{\beta}(g)
    =
    e^{ \beta (\real  \rho(g)-\rho(0))  }
    =
    e^{\beta (\real  \rho(-g)-\rho(0)) }
    =
    \phi_{\beta}(-g).
\end{equation}
For \( \omega \in \Omega^2(B_N,G) \) and \( \beta \geq 0 \) we define the \emph{activity} of $\omega$ by
\begin{equation*}
    \phi_{\beta}(\omega) \coloneqq \prod_{p \in C_2(B_N)}\phi_a\bigl(\omega(p)\bigr).
\end{equation*}

Note that for \( \sigma \in \Omega^1(B_N,G), \) the Wilson action lattice gauge theory probability measure can be written
\begin{equation}\label{eq: mubetakappaphi}
    \mu_{\beta,N}(\sigma) = \frac{\phi_{\beta} (d\sigma)}{\sum_{\sigma \in \Omega^1(B_N,G)} \phi_{\beta}(d\sigma)}.
\end{equation}
Moreover, in the case when $G = \mathbb{Z}^2$ for \( \omega \in \Omega^2(B_N,\mathbb{Z}_2),\) we have
\begin{equation*}
    \begin{split}
        &\phi_{\beta}(\omega) = \prod_{p \in C_2(B_N)} \phi_{\beta}\bigl( \omega(p) \bigr) = \prod_{p \in C_2(B_N)} e^{-2\beta \mathbb{1}\bigl( \omega(p) = 1\bigr)}
        =
        e^{-2\beta \sum_{p \in C_2(B_N)} \mathbb{1}\bigl( \omega(p) = 1\bigr)}
        =
        e^{-2\beta |\support \omega|}.
    \end{split}
\end{equation*}
We note that, by definition, if \( \omega,\nu \in \Omega^2(B_N,G)\) and \( \nu \) is a vortex in \( \omega,\) then \( \phi_\beta (\omega) = \phi_\beta(\nu) \phi_\beta(\omega-\nu).\)

We extend the notion of activity to vortex clusters \( \mathcal{V} \in \Xi\) by letting
\begin{equation*}
    \phi_\beta(\mathcal{V}) = \prod_{\nu \in \Lambda} \phi_\beta(\nu)^{n_{\mathcal{V}}(\nu)} = e^{-4\beta|\mathcal{V}|}.
\end{equation*}

\section{Low temperature cluster expansion}\label{sec: cluster expansions}
In this section we review the cluster expansion for the relevant Ising lattice gauge theory partition functions defined on a finite box $B_N$. The material here is for the most part well-known. See~\cite{fv2017} for a text book presentation for the standard Ising model and \cite{seiler} for a discussion in the context of lattice gauge theories.  

\subsection{Ursell function}
We will work with the  Ursell function corresponding to the choice of vortices as polymers and hard core interaction: two vortices are compatible if and only if they correspond to separate components in the graph $\mathcal{G}_2$.
Before defining the Ursell function, we need some additional notation. 
For \( k \geq 1,\) we write \( G \in \mathcal{G}^k\) if \( G\) is a connected graph with vertex set \( V(G) = \{ 1,2,\dots, k\}.\)  Let \( E(G)\) be the (undirected) edge set of \( G.\) 
Recall that we write $\nu_1 \sim \nu_2$ if there is \( p_1\in (\support \nu_1)^+\) and \( p_2 \in (\support \nu_2)^+\) such that \( p_1 \sim p_2\) in \( \mathcal{G}_2,\) where $\mathcal{G}_2$ was defined in Section~\ref{sec: the graphs}.

\begin{definition}[The Ursell function]\label{def:ursell}
    For \( k \geq 1 \) and  \( \nu_1,\nu_2,\dots , \nu_k \in \Lambda,\) we define
    \begin{equation*}
        U(  \nu_1, \ldots, \nu_k  )\coloneqq \frac{1}{k!} \sum_{G \in \mathcal{G}^k} (-1)^{|E(\mathcal{G} )|} \prod_{(i,j) \in E(\mathcal{G})} \mathbb{1}( \nu_i \sim \nu_j).
    \end{equation*} 
    Note that this definition is invariant under permutations of the vortices \( \nu_1,\nu_2,\dots,\nu_k.\)

   For \( \mathcal{V} \in \Xi_k,\) and any enumeration \( \nu_1,\dots, \nu_k \) (with multiplicities) of the vortices in \( \Xi_k,\) we define
   \begin{equation}\label{eq: ursell functions} 
        U(\mathcal{V}) = k! \, U(\nu_1,\dots, \nu_k).
   \end{equation}
\end{definition} 
Note that for any \( \mathcal{V} \in \Xi_1,\) we have \( U(\mathcal{V})=1,\) and for any \( \mathcal{V} \in \Xi_2,\)  we have
\( U(\mathcal{V}) = -1. \)

\subsection{Partition functions}
The partition function for Ising lattice gauge theory, viewed as a model for plaquette configurations ,can be written as follows:
\[
Z_{\beta,N} = \sum_{\omega \in \Omega^2_0(B_N,G)} e^{\beta \sum_{p \in C_2(B_N)} \real  (\rho(\omega(p))-\rho(0))} = \sum_{\omega \in \Omega^2_0(B_N,G)} \phi_\beta(\omega). 
\]
See, e.g., Section~3 of \cite{flv2020}. This is a finite sum and the definition extends to $\beta \in \mathbb{C}$.
An alternative representation for $ Z_{\beta,N}$ is given by the \emph{vortex partition function} which is defined by the following (formal) expression:
\begin{equation}\label{eq: vortex-partition}
     Z_{\beta,N}^v = \exp\left(
    \sum_{\mathcal{V}\in \Xi} \Psi_\beta(\mathcal{V})\right),
\end{equation}
where for \( \mathcal{V} \in \Xi, \) we define
\begin{equation*}
    \begin{split}
        \Psi_\beta(\mathcal{V}) \coloneqq U(\mathcal{V})
        \phi_\beta(\mathcal{V}),
    \end{split}
\end{equation*}
and $U$ is the Ursell function defined in Definition~\ref{def:ursell}.

It is not obvious that the series in the exponent of~\eqref{eq: vortex-partition} is convergent but this follows from the next lemma, assuming $\real  \beta > \beta_0(m)=\frac{1}{4} \log 30(m-2)$, and we verify below that in this case $\log Z_{\beta,N} = \log Z_{\beta,N}^v$.

\begin{lemma}\label{lemma: the assumption on a}Let $G=\mathbb{Z}_2$.
    Suppose $\real  \beta > \beta_0(m) $. Then, for any \( \nu \in \Lambda,\) we have 
    \begin{equation*}
        \sum_{\mathcal{V} \in \Xi \colon \nu \in \mathcal{V}} \bigl|\Psi_\beta(\mathcal{V})\bigr| \leq e^{|\support \nu|/3} \phi_\beta(\nu) 
    \end{equation*} 
    Moreover, the series in \eqref{eq: vortex-partition} is absolutely convergent. 
\end{lemma}
\begin{proof}
Let $\alpha = 1/3$. We will prove that for each \( \nu \in \Lambda\) we have 
    \begin{equation*}
        \sum_{\nu' \in \Lambda} |\phi_\beta(\nu')| e^{\alpha| \support \nu'|} \mathbb{1}(\nu \sim \nu') \leq \alpha |\support \nu| .
    \end{equation*}
    Given this, the conclusion follows from Theorem~5.4 of \cite{fv2017} by choosing ${a(\nu) \coloneqq \alpha |\support \nu |}$.
    Note that since \( M^2e^{-4 \real  \beta} < 1/3, \) we have
    \begin{equation*}
        M^2e^{-2(2 \real  \beta-\alpha)} < 1 
    \end{equation*}
    and
    \begin{equation*}
        \frac{\bigl(M^2e^{-2(2 \real  \beta-\alpha)} \bigr)^{2(m-1)}}{1-M^2e^{-2(4 \real  \beta-\alpha)} } \leq \alpha.
    \end{equation*}
    Thus, for any \( \nu \in \Lambda,\) we have 
    \begin{align*}
        &\sum_{\nu' \in \Lambda} |\phi_\beta(\nu')| e^{\alpha |\support \nu'|} \mathbb{1}(\nu \sim \nu') 
        =
        \sum_{\nu' \in \Lambda \colon \nu \sim \nu'} |\phi_\beta(\nu')| e^{\alpha |\support \nu'|}
        \\&\qquad=
        \sum_{\nu' \in \Lambda \colon \nu \sim \nu'} e^{-(2 \real \beta-\alpha) |\support \nu'|}
        \\&\qquad \leq 
        \sum_{p \in (\support \nu)^+} 
        \sum_{\substack{p'\in C_2(B_N)^+ {\colon}\\ p' \sim p}} \;
        \sum_{j=2(m-1)}^\infty
        |\Xi_{1,j,p'}|
        e^{-(2 \real \beta-\alpha) 2j}.
    \end{align*}
    Using Lemma~\ref{lemma: the number of vortices} and the definition of \(M,\) it follows that 
    \begin{align*}
        &\sum_{\nu' \in \Lambda} |\phi_\beta(\nu')| e^{\alpha |\support \nu'|} \mathbb{1}(\nu \sim \nu') \leq 
        \bigl| (\support \nu)^+\bigr| \sum_{j=2(m-1)}^\infty M^{2j}  e^{-(2 \real \beta-\alpha) 2j}
        \\&\qquad=
        |(\support \nu)^+| \frac{\bigl(M^2e^{-2(2\real \beta-\alpha)} \bigr)^{2(m-1)}}{1-M^2e^{-2(2\beta-\alpha)} }.
    \end{align*}
    The desired conclusion now follows from the choice of \( \alpha.\)
\end{proof}

\begin{lemma}\label{lem:partition-functions}
Let $G=\mathbb{Z}_2$. Suppose $\real \beta > \beta_0(m)$. Then
\begin{equation}\label{eq: log Z expansion}
    \log Z_{\beta,N} = \log Z_{\beta,N}^v = \sum_{\mathcal{V} \in \Xi} \Psi_\beta(\mathcal{V}),
\end{equation}
and \( \log Z_{\beta,N} \) is an analytic function of $\beta$.
\end{lemma}
\begin{proof}
The set $\Omega^2_0(B_N, G)$ is in bijection with the set of subsets of $\Lambda$ with the property that the vortices in each subset correspond to sets that are not connected in \( \mathcal{G}_2.\)
Therefore, we can write
\[
Z_{\beta,N} = \sum_{\Lambda' \subset \Lambda} \phi_{\beta}(\Lambda') \prod_{\{\nu,\nu'\} \subset \Lambda'} \mathbb{1}( \nu \nsim \nu')
\]
and this holds for any choice of $\beta$. On the other hand, if $\real \beta > \beta_0(m)$, we can apply Proposition~5.3 of \cite{fv2017} to see that the right-hand side in the last display equals $\log Z_{\beta,N}^v$ as defined in \eqref{eq: vortex-partition}. 
\end{proof}
We now assume \( \beta \) is real, and when $\beta > \beta_0$ we write $Z_{\beta,N}$ also for the vortex partition function.
We wish to express the Wilson loop expectation using the logarithm of the partition function. For this, we fix a loop $\gamma$ and an oriented surface $q$ such that $\gamma = \partial q$ and recall the following fact, see \cite[Section 3]{flv2020}. 
\begin{lemma}\label{lem:dec5}
    Let $G=\mathbb{Z}_2$. Let \( \beta \geq 0\) and let \( q\) be an oriented surface with \( \partial q = \gamma.\) Then for all $N$ such that $\support q \subseteq B_N$
    \begin{equation*}
       \mathbb{E}_{\beta,N}[W_\gamma] 
        = 
        Z^{-1}_{\beta,N} \sum_{\omega \in \Omega_0^2(B_N,G)}  \phi_\beta(\omega) \rho\bigl(\omega(q) \bigr).
    \end{equation*}
\end{lemma}

Consider now the weighted vortex partition function
\begin{equation}\label{eq:dec5.1}
   Z_{\beta,N}[q] \coloneqq \exp\left(\sum_{\mathcal{V} \in \Xi} \Psi_{\beta,q}(\mathcal{V})\right),
\end{equation}
where
\[
    \Psi_{\beta,q}(\mathcal{V}) \coloneqq \Psi_\beta(\mathcal{V}) \rho\bigl( \mathcal{V}(q) \bigr) =  U(\mathcal{V}) \phi_\beta(\mathcal{V})  
    \rho\bigl( \mathcal{V}(q) \bigr).
\]
The series on the right-hand side of~\eqref{eq:dec5.1} is absolutely convergent when $\beta > \beta_0(m)$ by the proof of Lemma~\ref{lemma: the assumption on a} since $ \bigl|\rho\bigl(\mathcal{V}(q)\bigr)\bigr| = 1$ for each \( \mathcal{V} \in \Xi.\) As in the proof of Lemma~\ref{lem:partition-functions}, using \cite[Proposition 5.3]{fv2017}, replacing the weight $\phi_\beta(\mathcal{V})$ by $\phi_\beta(\mathcal{V})\rho\bigl(\mathcal{V}(q)\bigr)$, we have
\[
\log Z_{\beta,N}[q] = \sum_{\omega \in \Omega_0^2(B_N,G)}  \phi_\beta(\omega) \rho\bigl(\omega(q) \bigr).
\] 
\begin{proposition}\label{proposition: the cluster expansion}
    Let $G=\mathbb{Z}_2$. Let $\beta > \beta_0(m)$ and let \( q \) be an oriented surface with \( \partial q = \gamma.\) 
    Then for all $N$ such that $\support q \subseteq B_N$,
    \begin{equation}\label{eq: expansion equation}
        -\log \mathbb{E}_{\beta,N}[ W_\gamma ] = \sum_{\mathcal{V} \in \Xi} \bigl( \Psi_\beta(\mathcal{V})-\Psi_{\beta,q}(\mathcal{V})\bigr) = 
        \sum_{\mathcal{V} \in \Xi} \Psi_\beta(\mathcal{V})  \pigl( 1-\rho \bigl(\mathcal{V}(q)\bigr) \pigr) .
    \end{equation}
\end{proposition}

\begin{proof}
    Using Lemma~\ref{lem:dec5} and then Lemma~\ref{lem:partition-functions} and \eqref{eq:dec5.1} we conclude that
    \[
        \log \mathbb{E}_{\beta,N}[W_\gamma] = \log \frac{Z_{\beta,N}[q]}{Z_{\beta,N}},
    \]
    which is what we wanted to prove.
\end{proof}
\begin{remark}
   Notice that Proposition~\ref{proposition: the cluster expansion} implies that $\mathbb{E}_{\beta,N}[W_\gamma] \in (0,1]$ when $\beta > \beta_0(m)$. This fact is not clear from the start since $W_\gamma \in \{-1,1 \}$ for every $\sigma \in \Omega^1(B_N,\mathbb{Z}_2) $. The positivity of $\mathbb{E}_{\beta,N}[W_\gamma]$ was pointed out in \cite{c2019} and proved there as a consequence of duality. Here we obtain the conclusion as a result of convergence of the cluster expansion.
\end{remark}

\section{Estimates for the cluster expansion}

Throughout this section, we assume that \( \gamma\) is a simple loop, and that \( q\) is an oriented surface with \( \partial q = \gamma.\) Recall that we use the notation  \( \ell = |\gamma| \). Recall also that \( \ell_c\) denotes the number of corners of $\gamma$, i.e., pairs of non-parallel edges in $\gamma$ that are both in the boundary of some common plaquette, similarly recall that \( \ell_{b}\) denotes the number of bottlenecks in $\gamma$, i.e., pairs \( (e,e')\) of parallel edges in $\gamma$ that are both in the boundary of some common plaquette. 
From now on, we also assume $G = \mathbb{Z}_2.$

The main goal of this section is to provide proofs of the following three propositions.
\begin{proposition}\label{proposition: m = 1 contribution new}
    Let $\beta >  \beta_0(m) .$  Further, let \( \beta^* \in (\beta_0(m) , \beta). \) Then 
    \begin{equation*}
        \begin{split}
            0 \leq \sum_{\mathcal{V} \in \Xi_1 } \bigl( \Psi_\beta(\mathcal{V})-\Psi_{\beta,q}(\mathcal{V}) \bigr) - \ell H(\gamma)  \leq D_1^* \ell e^{-16(m-1)\beta},
        \end{split}
    \end{equation*}
    where
    \[
    H_{m,\beta}(\gamma)=e^{-8(m-1)\beta}
            +
            \bigl(6(m-1)-\frac{2\ell_c-2\ell_{b}}{\ell}\bigr) e^{-4(4(m-1)-2)\beta}
    \]
    and \( D_1^* = D_1(\beta^*) = O_\beta(1)\) is defined in~\eqref{eq: D1 new}.
\end{proposition}

\begin{proposition}\label{proposition: m >=2 contribution new}
    Let $\beta >  \beta_0(m) .$  Further, let \( \beta^* \in \bigl(\beta_0(m) , \beta\bigr). \) Then 
     \begin{equation*}
        \begin{split}
            &
            \sum_{\mathcal{V} \in \Xi_{2^+} } \Bigl|  \Psi_\beta(\mathcal{V}) - \Psi_{\beta,q}(\mathcal{V}) \Bigr|
            \leq
            2D_1^* \ell e^{-16(m-1)\beta },
        \end{split}
    \end{equation*} 
    where \( D_1^* = D_1(\beta^*) = O_\beta(1)\) is defined in~\eqref{eq: D1 new}.
\end{proposition}

\begin{proposition}\label{proposition: uniformity in m' new}
    Let $\beta > \beta_0(m) $ and \( k \geq 1.\)  Further, let \( \beta^* \in \bigl(\beta_0(m) , \beta\bigr). \)  Then
    \begin{equation*}
        \begin{split}
            &
            \sum_{\mathcal{V} \in \Xi_{1^+,k^+}} \bigl| \Psi_\beta(\mathcal{V}) - \Psi_{\beta, q}(\mathcal{V}) \bigr|
            \leq 
            C_m  C_{\beta^*} \ell \sum_{j=0}^\infty 
            \max(1,j)^{m-1}  e^{-4(\beta-\beta^*) \max(k,\hat C_m(j+1))},
        \end{split}
    \end{equation*}  
    where \( C_\beta^* \) is defined in~\eqref{eq: cbeta*}.
\end{proposition}

Before giving the proofs of Propositions~\ref{proposition: m = 1 contribution new},~\ref{proposition: m >=2 contribution new},~and~\ref{proposition: uniformity in m' new}, we need several auxiliary results.

\begin{lemma}\label{lemma: m1min1 new}
    Let \( \beta \geq 0.\) Then 
    \begin{equation*}\label{eq: m1min} 
        \sum_{\mathcal{V} \in \Xi_{1,2(m-1),q} } \Psi_\beta(\mathcal{V}) 
        =
        \ell e^{-8(m-1)\beta}.
    \end{equation*}
\end{lemma}

\begin{proof}
    Let \( \{ \nu\} \in \Xi_{1,2(m-1)}. \) Then \( \Psi_\beta\bigl(\{\nu\}\bigr) = \phi_\beta(\nu) = e^{-8(m-1)\beta}.\) By Lemma~\ref{lemma: minimal vortices}, we have  \( \nu(q) \neq 1\) if and only if \( \nu\) is a minimal vortex around some edge \( e \in \gamma.\) Combining these observations, the  conclusion immediately follows.
\end{proof}

\begin{lemma}\label{lemma: m1min2 new}
    Let \( \beta \geq 0. \) Then 
    \begin{equation*}\label{eq: m1min2}
        \sum_{\mathcal{V} \in \Xi_{1,2(m-1)-2,q} } \Psi_\beta(\mathcal{V})
        =
        \bigl(6(m-1)\ell-2\ell_c-2\ell_{b}\bigr) e^{-4(4(m-1)-2)\beta}.
    \end{equation*}
\end{lemma}

\begin{proof}
    Let \( \nu\in \Lambda\) be such that \( |(\support \nu)^+| = 4(m-1)-2. \) Then \( \Psi_\beta\bigl( \{ \nu \} \bigr) = \phi_\beta(\nu) = e^{-4(4(m-1)-2)\beta}.\) 
    By Lemma~\ref{lemma: minimal 2}, there is \( \sigma \in \Omega^1(B_N,G)\) such that \( d \sigma = \nu\) and \(  (\support \sigma)^+  =\{ e_1,e_2 \} ,\) where \( e_1 \) and \( e_2\) are distinct edges and \( (\hat \partial e_1)^+ \cap (\hat \partial e_2)^+ \neq \emptyset.\) 
    Since
    \[
    \nu(q) = \sigma(\gamma) = \sum_{e \in \gamma} \sigma(e),
    \]
    and \( G = \mathbb{Z}_2, \) it follows that 
    \[
    \nu(q) = \begin{cases}
        1 &\text{if } \bigl|\support \gamma \cap \{ e_1,e_2 \}\bigr| = 1\cr0 &\text{else.}
    \end{cases} 
    \]
     From this it follows that
     \( \nu(q) \neq 1\) if and only if one of the following hold.
    \begin{enumerate}[label=(\roman*)]
        \item\label{item: 1st option} The edges \( e_1 \) and \( e_2 \) are parallel (see Figure~\ref{table: minimal configurations}(b)), and exactly one of \( e_1 \) and \( e_2 \) are in \( \support \partial q = \support \gamma.\)
        \item\label{item: 2nd option} The edges \( e_1 \) and \( e_2 \) are not parallel (see Figure~\ref{table: minimal configurations}(c)), and exactly one of \( e_1 \) and \( e_2 \) are in \( \support \partial q = \support \gamma.\)
    \end{enumerate}
    Noting that there is exactly~\( 2(m-1) \ell -2\ell_{b} \) distinct \( \nu \in \Lambda\) with \( |(\support \nu))^+| = 4(m-1)-2 \) such that~\ref{item: 1st option} holds, and exactly~\( 4(m-1)\ell-2\ell_c\) distinct \( \nu \in \Lambda\) with \( |(\support \nu))^+| = 4(m-1)-2 \) such that~\ref{item: 2nd option} holds, the conclusion immediately follows.
\end{proof}

\begin{lemma}\label{lemma: general Cbeta}
    Let $\beta > \beta_0(m)$ and \( p \in C_2(B_N).\) Then
    \begin{equation*}
        \sum_{\mathcal{V} \in \Xi_{1^+,1^+,p}}
        \bigl| \Psi_\beta(\mathcal{V}) \bigr|
        \leq
        \sum_{k=2(m-1)}^{\infty}
        M^{2k-1}
        e^{-2(2\beta-1/3)k}.
    \end{equation*}
\end{lemma}

\begin{proof}
    By Lemma~\ref{lemma: the assumption on a}, for any \( \nu \in \Lambda,\) we have 
    \begin{equation*}
        \sum_{\mathcal{V} \in \Xi \colon \nu \in \mathcal{V}} \Psi_\beta(\mathcal{V}) \leq e^{|\support \nu|/3} \phi_\beta(\nu),
    \end{equation*}
    and hence
    \begin{equation}\label{eq: b1}
        \begin{split}
            &\sum_{\mathcal{V} \in \Xi_{1^+,1^+,p}}
            \bigl| \Psi_\beta(\mathcal{V}) \bigr|
            \leq 
            \sum_{\{\nu\} \in \Xi_{1^+,1^+,p}}
            \sum_{\mathcal{V} \in \Xi \colon \nu \in \mathcal{V}} \bigl| \Psi_\beta(\mathcal{V}) \bigr|
            \leq 
            \sum_{\nu \in \Xi_{1^+,1^+,p}}
            e^{|\support \nu|/3} \phi_\beta(\nu) 
            \\&\qquad=
            \sum_{\nu \in \Xi_{1^+,1^+,p}}
            e^{-(2\beta-1/3)|\support \nu|}. 
        \end{split}
    \end{equation}
    Using Lemma~\ref{lemma: small 
    vortices}, it follows that 
    \begin{equation}\label{eq: b2}
        \begin{split}
            &
            \sum_{\nu \in \Xi_{1,1^+,p}}
            e^{-(4\beta-1/3)|\support \nu|}
            =
            \sum_{k=2(m-1)}^{\infty}
            \sum_{\nu \in \Xi_{1,k,p}}
            e^{-(2\beta-1/3)|\support \nu|}
            \\&\qquad=
            \sum_{k=2(m-1)}^{\infty}
            |\Xi_{1,k,p}|
            e^{-(2\beta-1/3) 2k}
            .
        \end{split}
    \end{equation}
    Combining~\eqref{eq: b1} and~\eqref{eq: b2} and using Lemma~\ref{lemma: the number of vortices}, the desired conclusion immediately follows.
\end{proof}

\begin{lemma}\label{lemma: new upper bound}
    Let \( \beta > \beta_0(m) ,\) \( k \geq 1\) and \( p \in C_2(B_N).\) Further, let \( \beta^* \in \bigl(\beta_0(m) , \beta\bigr). \) Then
    \begin{equation*}
        \begin{split}
            & \sum_{\mathcal{V} \in \Xi_{1^+,k^+,p}} \bigl| \Psi_\beta(\mathcal{V}) \bigr|   
            \leq  
            C_{\beta^*} e^{-4(\beta-\beta^*) k},
        \end{split}
    \end{equation*} 
    where \( C_{\beta^*}\) is defined by
    \begin{equation}\label{eq: cbeta*}
    C_{\beta^*} \coloneqq \sup_{N \geq 1} \sum_{\mathcal{V} \in \Xi_{1^+,1^+,p}} \bigl| \Psi_{\beta^*}(\mathcal{V}) \bigr| < \infty.
    \end{equation}
\end{lemma}

\begin{proof}
    By Lemma~\ref{lemma: general Cbeta}, we have \( C_{\beta^*}<\infty,\) and hence \( C_{\beta^*} \) is well defined.
    
    For any \( \mathcal{V}\in \Xi,\) we have \( \phi_\beta(\mathcal{V}) = e^{-2\beta| \mathcal{V}|}\) and \( \Psi_\beta(\mathcal{V}) = U(\mathcal{V}) \psi_\beta(\mathcal{V}),\) where \( U(\mathcal{V})\) does not depend on \( \beta, \) and hence \[
    \Psi_\beta(\mathcal{V}) = e^{-2(\beta-\beta^*)| \mathcal{V}|} \Psi_{\beta^*}(\mathcal{V}). \]
    Using this observation, we obtain
    \begin{equation*}
        \begin{split}
            & \sum_{\mathcal{V} \in \Xi_{1^+,k^+,p}} \bigl| \Psi_\beta(\mathcal{V}) \bigr|   
            \leq  
            e^{-4(\beta-\beta^*) k}\sum_{\mathcal{V} \in \Xi_{1^+,k^+,p}} \bigl| \Psi_{\beta^*}(\mathcal{V}) \bigr|     
            \leq  
            e^{-4(\beta-\beta^*) k}\sum_{\mathcal{V} \in \Xi_{1^+,1^+,p}} \bigl| \Psi_{\beta^*}(\mathcal{V}) \bigr|.
        \end{split}
    \end{equation*} 
    This concludes the proof.
\end{proof}

\begin{lemma}\label{lemma: new q bound}
    Let \( \beta > \beta_0(m),\) and let \( k \geq 1.\) Further, let \( \beta^* \in (\beta_0(m) , \beta). \) Then
    \begin{equation*}
        \begin{split}
            & \sum_{\mathcal{V} \in \Xi_{1^+,k^+,q}} \bigl| \Psi_\beta(\mathcal{V}) \bigr|   
            \leq  
            C_m  C_{\beta^*} \ell \sum_{j=0}^\infty 
            \max(1,j)^{m-1}  e^{-4(\beta-\beta^*) \max(k,\hat C_m(j+1))}.
        \end{split}
    \end{equation*} 
    where \( C_{\beta^*}\) is defined in~\eqref{eq: cbeta*}.
\end{lemma}

\begin{proof}
    Write
    \begin{equation}\label{eq: m  contribution 2a}
        \sum_{\mathcal{V} \in \Xi_{1,k^+,q}}   |\Psi_\beta(\mathcal{V})|  
        =
        \sum_{j = 0}^\infty
        \sum_{\substack{\mathcal{V} \in \Xi_{1^+,k^+,q} \\ \dist(\support \mathcal{V},\gamma)=j}}  |\Psi_\beta(\mathcal{V})|.
    \end{equation}
    Using Lemma~\ref{lemma: minimal size if at distance}, we can write
    \begin{equation}\label{eq: m  contribution 2b}
        \begin{split}
            &
            \sum_{j = 0}^\infty
            \sum_{\substack{\mathcal{V} \in \Xi_{1^+,k^+,q} \\ \dist(\support \mathcal{V},\gamma)=j}}  |\Psi_\beta(\mathcal{V})  |
            \leq
            \sum_{j = 0}^\infty
            \; \sum_{\substack{p \in C_2(B_N)^+ \mathrlap{\colon}\\ \dist(p,\gamma)=j}}
            \;\; \sum_{\mathcal{V} \in \Xi_{1^+,\max(k,\hat C_m(j+1))^+,p}}  |\Psi_\beta(\mathcal{V}) |
            .
        \end{split}
    \end{equation}
    By using Lemma~\ref{lemma: plaquettes on a given distance} and Lemma~\ref{lemma: new upper bound}, the right hand side of the previous equation can be bounded from above by
    \begin{align*}
        &\sum_{j=0}^\infty
        C_m C_{\beta^*}\ell \max(1,j)^{m-1} 
        e^{-4(\beta-\beta^*) \max(k,\hat C_m(j+1))}.
    \end{align*}
    Combining this observation with~\eqref{eq: m  contribution 2a}~and~\eqref{eq: m  contribution 2b}, we obtain the desired conclusion.
\end{proof}

\begin{remark}\label{remark: complex beta new}
    Lemmas~\ref{lemma: new upper bound}~and~\ref{lemma: new q bound} remain valid for complex $\beta$, assuming $\real  \beta > \beta_0(m)$. Indeed, the proofs work verbatim replacing $\beta$ by $\real  \beta$.
\end{remark}
    
We are now ready to give the proof of Proposition~\ref{proposition: m = 1 contribution new}.
\begin{proof}[Proof of Proposition~\ref{proposition: m = 1 contribution new}]
    Note first that 
    \begin{equation}\label{eq: m = 1 contribution 1 new}
        \sum_{\mathcal{V} \in \Xi_1 } \bigl( \Psi_\beta(\mathcal{V})-\Psi_{\beta, q}(\mathcal{V}) \bigr)
        =
        \sum_{\{ \nu \} \in \Xi_1} U\bigl(\{ \nu \}\bigr) \Psi_\beta(\nu) \pigl( 1 - \rho\bigl( \nu(q)\bigr) \pigr)
        =
        2\sum_{\{ \nu \} \in \Xi_{1,1^+,q} }\Psi_\beta(\nu).
    \end{equation} 
    Using Lemma~\ref{lemma: small vortices}, we get
    \begin{equation}\label{eq: m = 1 contribution 2a new}
        \begin{split}
        &\sum_{ \{ \nu\} \in \Xi_{1,1^+,q} }\Psi_\beta(\nu)
        = 
        \sum_{\{ \nu \} \in \Xi_{1,2(m-1),q}}\!\!\!\!
        \Psi_\beta(\nu)
        +\!\!\!\!
        \sum_{\{ \nu \} \in \Xi_{1,4(m-1)-2,q}} \Psi_\beta(\nu)
        +\!\!\!\!
        \sum_{\{ \nu \} \in \Xi_{1,4(m-1)^+,q}}\!\!\!\! 
        \Psi_\beta(\nu).
        \end{split}
    \end{equation} 
    Now note that if \( \mathcal{V} \in \Xi_1, \) then \( \Psi(\mathcal{V}) = \phi_\beta(\mathcal{V}) > 0. \) Using Lemma~\ref{lemma: new q bound}, applied with \( k = 4(m-1),\) we thus obtain
    \begin{equation}\label{eq: ineq eq in lemma new}
        \begin{split}
            & 0 \leq \sum_{\mathcal{V} \in \Xi_{1,k^+,q}} \Psi_\beta(\mathcal{V}) 
            =
            \sum_{\mathcal{V} \in \Xi_{1,4(m-1)^+,q}} \bigl| \Psi_\beta(\mathcal{V})  \bigr| 
            \leq 
            D_1^* \ell e^{-16(m-1)\beta },
        \end{split}
    \end{equation} 
    where
    \begin{equation}\label{eq: D1 new}
        D_1^* \coloneqq C_m  C_{\beta^*} e^{16(m-1)\beta^*}
        \sum_{j=0}^\infty 
            \max(1,j)^{m-1}  e^{-4(\beta-\beta^*) \max(0,\hat C_m(j+1)-4(m-1))}.
    \end{equation}
    We see that $D_1^* = O_\beta(1)$.
    At the same time, by combining Lemma~\ref{lemma: m1min1 new} and Lemma~\ref{lemma: m1min2 new}, we have
    \begin{equation}\label{eq: me last eq2 new}
        \begin{split}
            &%
            \sum_{\{ \nu \} \in \Xi_{1,2(m-1),q}}
            \Psi_\beta(\nu)
            +
            \sum_{\{ \nu \} \in \Xi_{1,4(m-1)-2,q}}
            \Psi_\beta(\nu)
            \\&\qquad= 
            \ell e^{-8(m-1)\beta}
            +
            \bigl(6(m-1)\ell-2\ell_c-2\ell_{b}\bigr) e^{-4(4(m-1)-2)\beta}=\ell H_{m,\beta}(\gamma).
        \end{split}
    \end{equation} 
    Combining~\eqref{eq: m = 1 contribution 1 new},~\eqref{eq: m = 1 contribution 2a new},~\eqref{eq: ineq eq in lemma new},~and~\eqref{eq: me last eq2 new}, we obtain the desired conclusion.
\end{proof}

\begin{proof}[Proof of Proposition~\ref{proposition: m >=2 contribution new}]
     Note first that given \( \mathcal{V} \in \Xi,\) we have \( \rho\bigl(\mathcal{V}(q)\bigr) \neq 1 \) if and only if  \( \rho\bigl(\mathcal{V}(q)\bigr) = -1 \) and hence \( \mathcal{V}(q) \neq 0. \)
     This implies in particular that
    \begin{equation*}
        \begin{split}
            &
            \sum_{\mathcal{V} \in \Xi_{2^+}} \Psi_\beta(\mathcal{V}) - \Psi_{\beta,q}(\mathcal{V}) 
            =
            \sum_{\mathcal{V} \in \Xi_{2^+}} \Psi_\beta(\mathcal{V})   \pigl( 1-\rho \bigl( \mathcal{V}(q) \bigr)\pigr)
            =
            2 
            \sum_{\mathcal{V}\in \Xi_{2^+,1^+,q}} \Psi_\beta(\mathcal{V}).
        \end{split}
    \end{equation*}
    Consequently, using Lemma~\ref{lemma: small vortices}, we find that
    \begin{equation}\label{eq: mdist abs ineq}
        \begin{split}
            &
            \sum_{\mathcal{V} \in \Xi_{2^+} } \Bigl|  \Psi_\beta(\mathcal{V}) - \Psi_{\beta, q}(\mathcal{V}) \Bigr|
            \leq
            2 \sum_{\mathcal{V}\in \Xi_{2^+,1^+,q}}  \bigl| \Psi_\beta(\mathcal{V}) \bigr| 
            \leq
            2 \sum_{\mathcal{V}\in \Xi_{1^+,4(m-1)^+,q}}  \bigl| \Psi_\beta(\mathcal{V}) \bigr| 
        \end{split}
    \end{equation}
    Using Lemma~\ref{lemma: new q bound}, applied with \( k = 4(m-1),\) we thus obtain
    \begin{equation}\label{eq: mdist abs ineq new}
        \begin{split}
            & \sum_{\mathcal{V} \in \Xi_{2^+} } \Bigl|  \Psi_\beta(\mathcal{V}) - \Psi_{\beta, q}(\mathcal{V}) \Bigr|
            \leq 2D_1^* \ell e^{-16(m-1)\beta },
        \end{split}
    \end{equation} 
    where \( D_1^*\) is given by~\eqref{eq: D1 new}. This concludes the proof.
\end{proof}

\begin{proof}[Proof of Proposition~\ref{proposition: uniformity in m' new}]
    By Lemma~\ref{lemma: small vortices}, without loss of generality, we can assume that \( k \geq 2(m-1).\)
    
     Given \( \mathcal{V} \in \Xi,\) if \( \rho\bigl(\mathcal{V}(q)\bigr) \neq 1 \) then \( \rho\bigl(\mathcal{V}(q)\bigr) = -1 \) and \(\mathcal{V}(q) \neq 0. \)
     This implies in particular that
    \begin{equation}\label{eq: mdist abs ineq 2}
        \begin{split}
            &
            \sum_{\mathcal{V} \in \Xi_{1^+,{k}^+}} \bigl| \Psi(\mathcal{V}) - \Psi_q(\mathcal{V})  \bigr| =
            \sum_{\mathcal{V}\in \Xi_{1^+,{k}^+}} \bigl| \Psi_\beta(\mathcal{V})  \bigr|  \pigl| 1- \rho\bigl(\mathcal{V}(q)\bigr) \pigr|
            =
            2 
            \sum_{\mathcal{V}\in \Xi_{1^+,{k}^+,q}} \bigl| \Psi_\beta(\mathcal{V}) \bigr| .
        \end{split}
    \end{equation} 
    Applying Lemma~\ref{lemma: new q bound}, the desired conclusion immediately follows.
\end{proof}

\section{Proof of Theorem~\ref{theorem: logZ}}

The main purpose of this section is provide a proof of Theorem~\ref{theorem: logZ}. To do this, we first state and prove two lemmas.  These are then combined with results from previous sections to yield a proof of Theorem~\ref{theorem: logZ}.

To simplify the notation in this section, we fix some \( p_0 \in \bigcap_{N \geq 1} C_2(B_N)^+.\) For \( \real \beta > \beta_0(m) \) and \( N \geq 1,\) we define
\begin{equation*}
    F_N(\beta) \coloneqq \sum_{\mathcal{V} \in \Xi_{1^+,1^+,p_0}(B_N)} \frac{\Psi_\beta(\mathcal{V})}{|\mathcal{V}|}
\end{equation*}

 
\begin{lemma}\label{eq: proposition: free energy}
    Let \( \real  \beta > \beta_0(m).\)  Then 
    \begin{equation*}
        \begin{split}
            \lim_{N \to \infty}\left|\frac{\log Z_{\beta,N}}{|C_2(B_N)^+|}  - F_N(\beta) \right| =0.
        \end{split} 
    \end{equation*}
    locally uniformly.
\end{lemma}

\begin{proof}
    Write $\Xi_{1^+,1^+,p_0} = \Xi_{1^+,1^+,p_0}(B_N)$. By~Lemma~\ref{lem:partition-functions} (using also Lemma~\ref{lemma: small vortices}), we have
    \begin{equation*}
        \begin{split}
            &\log Z_{\beta,N} = \sum_{\mathcal{V} \in \Xi} \Psi_\beta(\mathcal{V})
            =
            \sum_{p \in C_2(B_N)^+} \sum_{\mathcal{V} \in \Xi_{1^+,2(m-1)^+,p}} \frac{\Psi_\beta(\mathcal{V})}{|\mathcal{V}|}.
        \end{split}
    \end{equation*}

    Let \( k \geq 2(m-1)\) be arbitrary. Without loss of generality, can assume that \( N \) is large enough to ensure that \(\dist(p_0,\partial B_N)>k.\) Then, by Lemma~\ref{lemma: new upper bound}, using also Remark~\ref{remark: complex beta new}, it follows that for any \( p \in C_2(B_N),\) we have 
    \begin{equation}\label{eq: bounded and tail}
        \sum_{\mathcal{V} \in \Xi_{1^+,k^+,p}} \frac{\bigl| \Psi_\beta(\mathcal{V}) \bigr|  }{|\mathcal{V}|}
        \leq
        \frac{1}{k}\sum_{\mathcal{V} \in \Xi_{1^+,k^+,p}} \bigl| \Psi_\beta(\mathcal{V}) \bigr| 
        \leq
        \varepsilon(\beta,m,k) \coloneqq    C_{\beta^*} k^{-1}e^{-4(\beta-\beta^*) k},
    \end{equation}
    Note that
    \begin{equation*}
        \lim_{k \to \infty} \varepsilon(\beta,m,k) = 0
    \end{equation*}
    uniformly on compact subsets of the halfplane $\real \beta > \beta_0(m)$. 
    Now fix some \( p \in C_2(B_N)^+\) with \(\dist(p,\partial B_N)>k. \) 
    Since \(\dist(p_0,\partial B_N)>k \) and \(\dist(p,\partial B_N)>k, \) for each \( i <k \) there is a bijection \( \Xi_{1^+,i,{p_0}} \to \Xi_{1^+,i,p}\) which maps each \( \mathcal{V} \in \Xi_{1^+,i,p} \) to a translation of \(  \mathcal{V}\) in \( \Xi_{1^+,i,p}.\)
    Consequently, 
    \begin{align*}
        & \sum_{\mathcal{V} \in \Xi_{1^+,k^-,p} } \frac{\Psi_\beta(\mathcal{V})}{|\mathcal{V}|} 
        =
        \sum_{\mathcal{V} \in \Xi_{1^+,k-,p_0} } \frac{\Psi_\beta(\mathcal{V})}{|\mathcal{V}|} 
        ,
    \end{align*}
    and hence
    \begin{align*}
        &\biggl| \sum_{\mathcal{V} \in \Xi_{1^+,1^+,p}} \frac{\Psi_\beta(\mathcal{V})}{|\mathcal{V}|}
        -
        F_N(\beta)
        \biggr|
        =
        \biggl| \sum_{\mathcal{V} \in \Xi_{1^+,1^+,p}} \frac{\Psi_\beta(\mathcal{V})}{|\mathcal{V}|}
        -
        \sum_{\mathcal{V} \in \Xi_{1^+,1^+,p_0} } \frac{\Psi_\beta(\mathcal{V})}{|\mathcal{V}|}
        \biggr|
        \leq 2\varepsilon(\beta,m,j).
    \end{align*}

    Finally, we note that there is a constant \( C_m'\) such that
    \begin{equation}\label{eq: plaquettes close to boundary}
        \bigl\{ p \in C_2(B_N)^+ \colon \dist(p,\partial B_N) \leq k \bigr\} \leq C_m' k N^{m-1}.
    \end{equation}
    
    We now combine the above observations as follows.
    By~\eqref{eq: bounded and tail} and~\eqref{eq: plaquettes close to boundary}, we have
    \begin{equation*}
        \begin{split}
            &\biggl| \log Z_{\beta,N}  
            - \sum_{\substack{p \in C_2(B_N)^+ \mathrlap{\colon}\\ \dist(p,\partial B_N)>k}}\; \sum_{\mathcal{V} \in \Xi_{1^+,1^+,p}} \frac{\Psi_\beta(\mathcal{V})}{|\mathcal{V}|}
            \biggr|
            \leq C_m' k N^{m-1} \varepsilon(\beta,m,1).
        \end{split}
    \end{equation*}
    Using~\eqref{eq: bounded and tail}, it follows that
    \begin{equation*}
        \begin{split}
            &\biggl| \log Z_{\beta,N}  
            - \sum_{\substack{p \in C_2(B_N)^+ \mathrlap{\colon}\\ \dist(p,\partial B_N)>k}}\; F_N(\beta)
            \biggr|
            \leq 2\varepsilon(\beta,m,k)\bigl|C_2(B_N)^+\bigr| + C_m' k N^{m-1} \varepsilon(\beta,m,1).
        \end{split}
    \end{equation*}
    Again using~\eqref{eq: bounded and tail} and~\eqref{eq: plaquettes close to boundary}, we get
    \begin{equation*}
        \begin{split}
            &\biggl| \log Z_{\beta,N}  
            - \bigl|C_2(B_N)^+\bigr| F_N(\beta)
            \biggr|
            \leq 2\varepsilon(\beta,m,k)\bigl|C_2(B_N)^+\bigr| + 2C_m' k N^{m-1} \varepsilon(\beta,m,1).
        \end{split}
    \end{equation*}
    Dividing both sides by \( |C_2(B_N)^+|\) and letting \( N \to \infty,\) we finally obtain
    \begin{equation*}
        \begin{split}
            &\lim_{N \to \infty}\biggl| \frac{\log Z_{\beta,N}}{|C_2(B_N)^+|}  
            -  F_N(\beta)
            \biggr|
            \leq 2\varepsilon(\beta,m,k)
        \end{split}
    \end{equation*}
    and this bound is decreasing in $\real \beta > \beta_0$.
    Since \( k \) was arbitrary, the desired conclusion follows.
\end{proof}

 \begin{lemma}\label{lemma: convergence of cluster partition}
    Let \( \real \beta > \beta_0(m). \)   Then \( F_N(\beta) \) 
    converges as \( N \to \infty\) locally uniformly.
\end{lemma}

\begin{proof}
    Let \( k \geq 1. \) Then we can write 
    \begin{equation*}
        F_N(\beta) 
        =
        \sum_{\mathcal{V} \in \Xi_{1^+,k^-,p_0}(B_N)} \frac{\Psi_\beta(\mathcal{V})}{|\mathcal{V}|}
        +
        \sum_{\mathcal{V} \in \Xi_{1^+,k^+,p_0}(B_N)} \frac{\Psi_\beta(\mathcal{V})}{|\mathcal{V}|}.
    \end{equation*}
    Note that if \( j \leq N, \) then we have \( \Xi_{1^+,j,p_0}(B_N) = \Xi_{1^+,j,p_0}(B_j). \)
    By Lemma~\ref{lemma: new upper bound} and Remark~\ref{remark: complex beta new}, we have
    \begin{equation*}
        \begin{split}
            &\biggl| \sum_{\mathcal{V} \in \Xi_{1^+,k^+,p_0}(B_N)} \frac{\Psi_\beta(\mathcal{V})}{|\mathcal{V}|} \biggr|
            \leq 
            \frac{1}{k}\sum_{\mathcal{V} \in \Xi_{1^+,k^+,p_0}(B_N)}  \bigl|\Psi_\beta(\mathcal{V})\bigr| 
            \leq 
            C_{\beta^*} k^{-1} e^{4(\beta-\beta^*)k}.
        \end{split}
    \end{equation*}
    Note in particular that this upper bound does not depend on \( N,\) is decreasing in \( \real \beta,\) and tends to zero as \( k' \to \infty.\) From this the desired conclusion immediately follows.  
\end{proof}

\begin{proof}[Proof of Theorem~\ref{theorem: logZ}] 
 As in Lemma~\ref{lemma: convergence of cluster partition}, write $F_N(\beta)= \sum_{\mathcal{V} \in \Xi_{1^+,1^+,p_0}(B_N)} \Psi_\beta(\mathcal{V})/|\mathcal{V}|$. For each $N$, by Lemma~\ref{lem:partition-functions} function is analytic in the half plane $\real \beta > \beta_0(m)$ and by Lemma~\ref{lemma: convergence of cluster partition} it converges locally uniformly as $N\to \infty$ to a limiting function which is also analytic. By Lemma~\ref{eq: proposition: free energy}, $\log Z_{\beta,N}$ also converges locally uniformly to the same limit. On the other hand,
    if \( q = 1 \cdot p\) then, using Lemma~\ref{lemma: m1min1 new}, we have  
    \begin{equation} \label{mar23.1}
        \sum_{\mathcal{V} \in \Xi_{1,2(m-1),p}(B_N) } \phi_\beta(\mathcal{V})  
        =
        \sum_{\mathcal{V} \in \Xi_{1,2(m-1),q} } \phi_\beta(\mathcal{V})  
        =
        4e^{-8(m-1)\beta},
    \end{equation}  
    for all $N$ sufficiently large. Similarly, using Lemma~\ref{lemma: m1min2 new}, we have
    \begin{equation}  \label{mar23.2}
        \sum_{\mathcal{V} \in \Xi_{1,2(m-1)-2,p}(B_N) } \phi_\beta(\mathcal{V})
        =
        \sum_{\mathcal{V} \in \Xi_{1,2(m-1)-2,q}(B_N) } \phi_\beta(\mathcal{V}) 
        = 
        \bigl(24(m-1) -16\bigr) e^{-4(4(m-1)-2)\beta}.
    \end{equation}  
    On the other hand, by Lemma~\ref{lemma: new upper bound}, and Remark~\ref{remark: complex beta new} we have
    \begin{align} \label{mar23.3}
        &\sum_{\mathcal{V} \in \Xi_{1^+,4(m-1)^+,p}(B_N)}
        \bigl| \Psi_\beta(\mathcal{V}) \bigr|
        \leq C_{\beta^*} e^{-4(\beta-\beta^*)k}
    \end{align}
    for all $N$. We conclude by combining \eqref{mar23.1}, \eqref{mar23.2}, and \eqref{mar23.3}.
\end{proof}

\section{Proof of Theorem~\ref{theorem: main theorem}}

\begin{proof}[Proof of Theorem~\ref{theorem: main theorem}]
    By combining Proposition~\ref{proposition: m = 1 contribution new} and Proposition~\ref{proposition: m >=2 contribution new}, we obtain
    \begin{equation*}
        \begin{split}
            &
            -
            3D_1^* \ell e^{-16(m-1)\beta} 
            \leq 
            \ell H(\ell) +  
            \sum_{\mathcal{V}\in \Xi} \bigl( \Psi_{\beta,q}(\mathcal{V}) - \Psi_\beta(\mathcal{V})\Bigr)
            \leq
            2D_1^* \ell  e^{-16\beta (m-1) } .
        \end{split}
    \end{equation*}
    Using  Proposition~\ref{proposition: the cluster expansion}, the proof is complete.
\end{proof}

\section{Proof of Theorem~\ref{theorem: limit of ratio exists 2} and Theorem~\ref{theorem: limit of ratio exists}}

In this section, we state and prove Proposition~\ref{eq: proposition rectangle limit 1} and Proposition~\ref{eq: proposition rectangle limit 2}, which are the more technical versions of the two main results Theorem~\ref{theorem: limit of ratio exists 2} and Theorem~\ref{theorem: limit of ratio exists}. The main tool in their proofs is Lemma~\ref{lemma: general limit bound 0} below. 

We now introduce some additional notation.
Fix a mapping \( \nu \mapsto \sigma^\nu\) from \( \Lambda \) to \(\Omega^1(B_N,G) \) which satisfies the following.
\begin{enumerate}
    \item For each \( \nu \in \Lambda,\) we have \( d \sigma^\nu = \nu .\)
    \item For each \( \nu \in \Lambda,\) the support of \( \sigma^\nu\) is contained in the smallest box \( B_{\mathcal{V}}\) that contains the support of \( \nu.\) 
    \item If \( \tau \) is a translation or rotation of the lattice with the property that \( \support \nu \circ \tau \subseteq C_2(B_N),\) then \( \sigma^{\nu \circ \tau} = \sigma^{\nu}\circ \tau.\)
\end{enumerate}
Note that such a mapping exists by Lemma~\ref{lemma: poincare}.

Given \( \mathcal{V} \in \Xi,\) let \( E_{\mathcal{V}} = \bigcup_{\nu \in \mathcal{V}} \support \sigma^\nu.\)
Given an edge \( e \in C_1(B_N)^+\) and \( m,m' \geq 1,\) let 
\begin{equation*}
    \Xi_{m,m',e} \coloneqq \bigl\{ \mathcal{V} \in \Xi_{m,m'} \colon e \in E_{\mathcal{V}} \bigr\}.
\end{equation*}
Define \( \Xi_e,\) \( \Xi_{m+,m',e}, \) \( \Xi_{m,m'+,e}, \)  and \( \Xi_{m+,m'+,e} \) as before. Finally, we let \( \Xi_{1^+,m'-,e} = \Xi_{1^+,1^+,e} \smallsetminus \Xi_{1^+,m'+,e}. \)

Fix any \( e_0 \in \bigcap_{N=1}^\infty C_1(B_{N})\) and let \( \gamma^0\) be the bi-infinite line through \( e_0\).
Given \( R \geq 1,\) let \( \hat \gamma^0\) be an axis-parallel translation of \( - \gamma^0\) such that the distance between \( \gamma^0\) and \( \hat \gamma^0\) is \( R,\) and let \( \gamma^R = \gamma^0 + \hat \gamma^0. \)
Let \( q^R\) be the bi-infinite strip with boundary \( \gamma^R,\) and let \( q^0\) be a half-plane with boundary \( \gamma^0\) (see~Figure~\ref{figure: biinfinite paths}).  We use the same notations for the restrictions of \( \gamma^0,\) \( \gamma^R,\) \( q^0,\) and \( q^R\) to \( C_1(B_N)\) and \( C_2(B_N)\)  respectively.
\begin{figure}[ht]
    \centering 
    \begin{subfigure}{.45\textwidth}\centering
        \begin{tikzpicture}[scale=0.24]

        \foreach \x in {0,...,23}{
            \foreach \y in {0,...,11}{
                \fill[fill=detailcolor07, fill opacity=0.18] (\x+0.02,\y+0.02) -- (\x+0.02,\y+1-0.02) -- (\x+1-0.02,\y+1-0.02)  -- (\x+1-0.02,\y+0.02);
            }
        } 
                  
        \draw[->] (0,0) -- (12,0) node[below] {\(\gamma^0\)}  -- (24,0);  

        \draw (12,6) node {\( q^0\)};
 
    \end{tikzpicture}
    \caption{The path \( \gamma^0 \) and the oriented surface \( q^0 \) (purple plaquettes).}
    \end{subfigure}
    \hspace{2em}
    \begin{subfigure}{.45\textwidth}\centering
        \begin{tikzpicture}[scale=0.24]

        \foreach \x in {0,...,23}{
            \foreach \y in {0,...,7}{
                \fill[fill=detailcolor07, fill opacity=0.18] (\x+0.02,\y+0.02) -- (\x+0.02,\y+1-0.02) -- (\x+1-0.02,\y+1-0.02)  -- (\x+1-0.02,\y+0.02);
            }
        } 
                  
        \draw[->] (0,0) -- (12,0) node[below] {\(\gamma^0\)}  -- (24,0);

        \draw[<-] (0,8) -- (12,8) node[above] {\(\hat \gamma^0\)}  -- (24,8); 

        \draw [decorate,
    decoration = {brace,amplitude=6pt,mirror},detailcolor00] (25,0) --  (25,8) node[midway, left, xshift=20pt] {\color{black}\( R\)}; 

    \draw (12,4) node {\( q^R\)};

    \end{tikzpicture}
    \caption{The 1-chain \( \gamma^R = \gamma^0 + \hat \gamma^0\) and the oriented surface \(q^R\) (purple plaquettes).}
    \end{subfigure}
    \caption{In the figures above, we illustrate \( \gamma^0 \) and \( \gamma^R,\) as well as the oriented surfaces \( q^0 \) and \( q^R.\)}
    \label{figure: biinfinite paths}
\end{figure}

Given a rectangular loop \( \gamma \) and  \( j \geq 1, \) let \( \gamma_{c,j} \) be the restriction of \( \gamma \) to the set of edges that are on distance at most \( j \) from a corner of \( \gamma \) (see Figure~\ref{figure: gammacj}). Note that if \( \gamma\) is a rectangular loop, then \( |\gamma_{c,j}| \leq 8j.\)

\begin{figure}[ht]
    \centering
    \hfill
    \begin{subfigure}[b]{0.45\textwidth}
    \centering
    \begin{tikzpicture}[scale=0.2] 
                 
        \draw[->] (12,0) -- (24,0) -- (24,6);
        \draw[->] (24,6) -- (24,12) -- (12,12);
        \draw[->] (12,12) -- (0,12) -- (0,6);
        \draw[->] (0,6) -- (0,0) -- (12,0);
 
    \end{tikzpicture}
    \caption{The loop \( \gamma.\)}
    \end{subfigure}
    \hspace{2em}
    \begin{subfigure}[b]{0.45\textwidth}
    \centering
    \begin{tikzpicture}[scale=0.2]  
                 
        \draw[->] (21,0) -- (24,0) -- (24,3);
        \draw[->] (24,9) -- (24,12) -- (21,12);
        \draw[->] (3,12) -- (0,12) -- (0,9);
        \draw[->] (0,3) -- (0,0) -- (3,0); 
 
    \end{tikzpicture}
    \caption{The path \( \gamma_{c,j}.\)}
    \end{subfigure} 
    \hfill
    
    \caption{In the figures above, we draw a loop \( \gamma \) and a corresponding path \( \gamma_{c,j}.\)}
    \label{figure: gammacj}
\end{figure}

\begin{lemma}\label{lemma: general limit bound 0}
    Let \( \real  \beta > \beta_0(m).\) 
    Let \( \gamma\) be a rectangular loop with axis-parallel sides with lengths \( R\) and \( T\), respectively, where \( R \leq T.\) Let \( q \) be the unique flat oriented surface with \( \partial q = \gamma.\) 
    Let  \( k \leq R\) and \( k' \leq T.\)  
    Assume that \( N \) is large enough to ensure that  $\support \gamma \subset C_1(B_{N-k}).$ 
    Then
    \begin{equation*}
            \begin{split} 
            &\biggl| \sum_{\mathcal{V} \in \Xi_{1^+,k'-}} \Psi_\beta(\mathcal{V}) \pigl( 1-\rho \bigl(\mathcal{V}(q)\bigr) \pigr)
            \\&\qquad\qquad-  
            |\gamma | \sum_{\mathcal{V} \in \Xi_{1^+,k-,e_0}} \bigl|E_{\mathcal{V}}  \cap \support \gamma^0 \bigr|^{-1} \Psi_\beta(\mathcal{V}) \pigl( 1-\rho \bigl(\mathcal{V}(q^0)\bigr) \pigr) 
            \\&\qquad\qquad-  
            |\gamma | \sum_{\mathcal{V} \in \Xi_{1^+,k'-,e_0}\smallsetminus\Xi_{1^+,k-,e_0}} \bigl|E_{\mathcal{V}}  \cap \support \gamma^R \bigr|^{-1} \Psi_\beta(\mathcal{V})\pigl( 1-\rho \bigl(\mathcal{V}(q^R)\bigr) \pigr)  
            \biggr|
            \\&\qquad\leq   
            4\sum_{j=0}^{k'-1}|\gamma_{c,j} | \sum_{\mathcal{V} \in \Xi_{1^+,j,e_0}} \bigl| \Psi_\beta(\mathcal{V}) \bigr| .
            \end{split}
        \end{equation*} 
\end{lemma}

\begin{proof}
    Note that, since \( \gamma \) is a rectangular loop and \( q \) is a flat oriented surface with boundary \( \gamma, \) we have $\support q \subset B_{N-k}.$

    Fix \( j \in \{ 1,2,\dots, k'-1 \}.\) Then 
    \begin{equation}\label{eq: gamma split part 2}
        \begin{split}
            &\sum_{\mathcal{V} \in \Xi_{1^+,j}} \Psi_\beta(\mathcal{V}) \pigl( 1-\rho \bigl(\mathcal{V}(q)\bigr) \pigr) 
            =
            \sum_{e \in \gamma}\sum_{\mathcal{V} \in \Xi_{1^+,j,e}} |E_{\mathcal{V}} \cap \support \gamma|^{-1} \Psi_\beta(\mathcal{V}) \pigl( 1-\rho \bigl(\mathcal{V}(q)\bigr) \pigr) 
            \\&\qquad=
            \sum_{e \in \gamma_{c,k'}}\sum_{\mathcal{V} \in \Xi_{1^+,j,e}} |E_{\mathcal{V}}  \cap \support \gamma|^{-1} \Psi_\beta(\mathcal{V}) \pigl( 1-\rho \bigl(\mathcal{V}(q)\bigr) \pigr) 
            \\&\qquad\qquad+
            \sum_{e \in \gamma\smallsetminus \gamma_{c,k'}}\sum_{\mathcal{V} \in \Xi_{1^+,j,e}} \bigl|E_{\mathcal{V}}  \cap \support \gamma \bigr|^{-1} \Psi_\beta(\mathcal{V}) \pigl( 1-\rho \bigl(\mathcal{V}(q)\bigr) \pigr) .
            \end{split}
        \end{equation} 
        Since \( j<k' \leq R \leq T , \) for any \( e \in \gamma\smallsetminus \gamma_{c,j}, \) we have 
        \begin{equation}\label{eq: gammaR part}
            \begin{split}
            &
            \sum_{\mathcal{V} \in \Xi_{1^+,j,e}} \bigl| E_{\mathcal{V}}  \cap \support \gamma\bigr|^{-1} \Psi_\beta(\mathcal{V}) \pigl( 1-\rho \bigl(\mathcal{V}(q)\bigr) \pigr) 
            \\&\qquad=
            \sum_{\mathcal{V} \in \Xi_{1^+,j,e_0}} \bigl|E_{\mathcal{V}}  \cap \support \gamma^R  \bigr|^{-1} \Psi_\beta(\mathcal{V}) \pigl( 1-\rho \bigl(\mathcal{V}(q^R )\bigr) \pigr) .
            \end{split}
        \end{equation} 
        Combining~\eqref{eq: gamma split part 2}~and~\eqref{eq: gammaR part}, we obtain
        \begin{equation}\label{eq: first useful ineq}
            \begin{split} 
            &\biggl| \sum_{\mathcal{V} \in \Xi_{1^+,k'-}} \Psi_\beta(\mathcal{V}) \pigl( 1-\rho \bigl(\mathcal{V}(q)\bigr) \pigr)
            -  
            |\gamma | \sum_{\mathcal{V} \in \Xi_{1^+,k'-,e_0}} \bigl|E_{\mathcal{V}}  \cap \support \gamma^R \bigr|^{-1} \Psi_\beta(\mathcal{V}) \pigl( 1-\rho \bigl(\mathcal{V}(q^R)\bigr) \pigr)  
            \biggr|
            \\&\qquad\leq   
            4\sum_{j=0}^{k'-1}|\gamma_{c,j} | \sum_{\mathcal{V} \in \Xi_{1^+,j,e_0}} \bigl| \Psi_\beta(\mathcal{V}) \bigr|   .
            \end{split}
        \end{equation}
        Finally, we note that since  \( k \leq R,\)  for any \( \mathcal{V} \in \Xi_{1^+,k-,e_0} \) we have
        \begin{equation}\label{eq: gamma0 part 1}
            \bigl|E_{\mathcal{V}}  \cap \support \gamma^R  \bigr|
            =
            \bigl|E_{\mathcal{V}}  \cap \support \gamma^0  \bigr|
        \end{equation}
        and
        \begin{equation}\label{eq: gamma0 part 2}
            \mathcal{V}(q^R ) 
            =
            \mathcal{V}(q^0). 
        \end{equation}
        Combining~\eqref{eq: first useful ineq},~\eqref{eq: gamma0 part 1}~and~\eqref{eq: gamma0 part 2}, we obtain the desired conclusion.
\end{proof}

\begin{lemma}\label{lemma: upper bound on term}
    For any \( k \geq 1\) and any \( e \in C_1(B_N)^+,\) we have
    \begin{equation}\label{eq: upper bound on term 1}
        \begin{split}
            & \sum_{\mathcal{V} \in \Xi_{1^+,k^+,e}} \bigl| \Psi_\beta(\mathcal{V}) \bigr|  
            \leq
            2C_{\beta^*} \binom{m}{2} \sum_{j=k}^\infty
            (j+1)^m  e^{-4(\beta-\beta^*)j} 
        \end{split}
    \end{equation}  
    and
    \begin{equation}\label{eq: upper bound on term 2}
        \sum_{j=0}^{k-1} j \sum_{\mathcal{V} \in \Xi_{1^+,j,e}} \bigl| \Psi_\beta(\mathcal{V}) \bigr|   \leq 
        2C_{\beta^*} \binom{m}{2} \sum_{j=2(m-1)}^{k-1} j 
        \sum_{i=j}^\infty
         (i+1)^m  e^{-4(\beta-\beta^*)i} .
    \end{equation}
\end{lemma}

\begin{proof}
    Let \( j \geq k, \) and let \( P_j\) be the set of all positively oriented plaquettes that are on distance at most \( j \) from \( e.\) 
    If \( \mathcal{V} \in \Xi_{1^+,j,e},\) then we must have \( \support \mathcal{V} \cap P_j \neq \emptyset.\)  From this it follows that
    \begin{equation*} 
        \sum_{\mathcal{V} \in \Xi_{1^+,k^+,e}} \bigl| \Psi_\beta (\mathcal{V}) \bigr|    
        \leq   
        \sum_{j=k}^\infty
        \sum_{p \in P_j} \sum_{\mathcal{V} \in \Xi_{1^+,j^+,p}}  
        \bigl| \Psi_\beta(\mathcal{V}) \bigr| .
    \end{equation*} 
    Note that  \( |P_j| \leq \binom{m}{2}2(j+1)^m.\) Using Lemma~\ref{lemma: plaquettes on a given distance} and Lemma~\ref{lemma: new upper bound}, we thus obtain~\eqref{eq: upper bound on term 1}.
    Finally, using first Lemma~\ref{lemma: small vortices}, we note that. 
    \begin{align*} 
        &\sum_{j=0}^{k-1} j \sum_{\mathcal{V} \in \Xi_{1^+,j,e}} \bigl| \Psi_\beta (\mathcal{V}) \bigr|   
        \leq
        \sum_{j=2(m-1)}^{k-1} j \sum_{\mathcal{V} \in \Xi_{1^+,j^+,e}} \bigl| \Psi_\beta (\mathcal{V}) \bigr|.
    \end{align*}  
    Using~\eqref{eq: upper bound on term 1}, we obtain~\eqref{eq: upper bound on term 1} as desired. This concludes the proof. 
\end{proof}

We now state and prove Proposition~\ref{eq: proposition rectangle limit 1} and Proposition~\ref{eq: proposition rectangle limit 2}, which are the more technical versions of the two main results Theorem~\ref{theorem: limit of ratio exists 2} and Theorem~\ref{theorem: limit of ratio exists}.
\begin{proposition}\label{eq: proposition rectangle limit 1}
    Let \( \real  \beta > \beta_0(m).\) 
    Let  \( (R_n)_{n \ge 1} \) and \( (T_n)_{n \ge 1}\) be non-decreasing sequences of positive integers with \( \lim_{n \to \infty} \min (R_n,T_n) = \infty.\) For each \( n \geq 1,\) let \( \gamma_n\) be a rectangular loop with axis-parallel sides with lengths \( R_n\) and \( T_n\), respectively.
    Then the limit \( \lim_{n\to \infty}    -\log \langle W_{\gamma_n} \rangle_{\beta} /|\gamma_n| \) exists and is given by
    \begin{equation*}
        \begin{split}
         &\hat V_\beta \coloneqq  V_\beta/2 \coloneqq 
         \lim_{N \to \infty}\sum_{\mathcal{V} \in \Xi_{1^+,1^+,e_0}} \bigl|E_{\mathcal{V}} \cap \support \gamma^0\bigr|^{-1} \Psi_\beta(\mathcal{V}) \pigl( 1-\rho \bigl(\mathcal{V}(q^0)\bigr) \pigr).
         \end{split}
    \end{equation*} 
    Moreover, for any \( n \geq 1, \) we have
    \begin{align*}
        \biggl| \frac{-\log \langle W_{\gamma_n} \rangle_\beta }{|\gamma_n|} - \hat V_\beta \biggr| 
        \leq& 
        32 m^2 |\gamma_n|^{-1} \sum_{j=2(m-1)}^\infty (3M)^{2j-1} (j+1)^{m+1}2^{j/(2(m-1))}e^{-4\beta j}
        \\&\quad+
        4m^2 \sum_{j=\min(R_n,T_n)}^\infty (3M)^{2j-1} (j+1)^m2^{j/(2(m-1))}e^{-4\beta j} .
    \end{align*}
\end{proposition}

\begin{proof}[Proof of Proposition~\ref{eq: proposition rectangle limit 1}]
    Note that for any \( k \geq 1 \) and \( N \) large enough to ensure that \(\dist(e_0,\partial B_N)>k,\) we have \( \Xi_{1^+,k,e_0}(B_N) = \Xi_{1^+,k,e_0}(B_k).\) 
    By Lemma~\ref{lemma: upper bound on term}, it follows that the sum 
    \begin{align*} 
         &\hat V_{\beta,N} \coloneqq  \sum_{\mathcal{V} \in \Xi_{1^+,k^+,e_0}(B_N)} \bigl| E_{\mathcal{V}}  \cap \support \gamma^0  \bigr|^{-1} \Psi_\beta(\mathcal{V}) \pigl( 1-\rho \bigl(\mathcal{V}(q^0)\bigr) \pigr)  
    \end{align*}
    is well defined and absolute convergent, uniformly in \( N, \) and hence \( \hat V_\beta\) is well defined.

    Fix \( n \geq 1.\) Let \( q_n\) be the unique 2-form with \( \partial q_n = \gamma_n\) that minimizes \( |\support q_n|.\)
    By Proposition~\ref{proposition: the cluster expansion}, we have
    \begin{equation*}
        \begin{split}
            &-\log \mathbb{E}[ W_{\gamma_n} ]_{\beta,N} 
            =
            \sum_{\mathcal{V} \in \Xi} \Psi_\beta(\mathcal{V}) \pigl( 1-\rho \bigl(\mathcal{V}(q_n)\bigr) \pigr)
            =
            \sum_{j=1}^{\infty}  \sum_{\mathcal{V} \in \Xi_{1^+,j}} \Psi_\beta(\mathcal{V}) \pigl( 1-\rho \bigl(\mathcal{V}(q_n)\bigr) \pigr)  .
        \end{split}
    \end{equation*}
    Let \( (k_n)_{n\geq 1} \) be a sequence of non-negative integers such that for each \( n \geq 1, \) \( {k_n \leq  \min(R_n,T_n),}\) and \( \lim_{n\to \infty} k_n/|\gamma_n| = 0.\) 
    Then, for each \( n \geq 1, \) by applying Lemma~\ref{lemma: general limit bound 0} with \(k=k'=k_n ,\) we obtain 
    \begin{equation*}
            \begin{split} 
            &\bigl| -\mathbb{E}[ W_{\gamma_n} ]_{\beta,N}   - |\gamma| \hat V_{\beta,N}
            \bigr| 
            \leq   
            4\sum_{j=1}^{k_n-1} |\gamma_{c,j} | 
            \sum_{\mathcal{V} \in \Xi_{1^+,j,e_0}} \bigl| \Psi_\beta(\mathcal{V}) \bigr| 
            +
            4|\gamma | \sum_{\mathcal{V} \in \Xi_{1^+,k_n+,e_0}} \bigl| \Psi_\beta(\mathcal{V}) \bigr|   .
        \end{split}
    \end{equation*}  
    Using Lemma~\ref{lemma: small vortices}, we note that for any \( j<2(m-1), \) we have \(\Xi_{1^+,k_n-,e_0} = \emptyset. \) 
    Also, we note that since \( \gamma_n\) is rectangular for each \( n \geq 1, \) we have \( |\gamma_{c,j}| \leq 8j\) for each \( j \geq 1. \)
    Using Lemma~\ref{lemma: upper bound on term}, we thus obtain
    \begin{align*}
        &\biggl| \frac{-\log \mathbb{E}[ W_{\gamma_n} ]_{\beta,N}  }{|\gamma_n|} 
        -
        \sum_{\mathcal{V} \in \Xi_{1^+,1^+,e_0}} \bigl|E_{\mathcal{V}} \cap \support \gamma^0\bigr|^{-1} \Psi_\beta(\mathcal{V}) \pigl( 1-\rho \bigl(\mathcal{V}(q^0)\bigr) \pigr)
        \biggr|
        \\&\qquad \leq    
        4 |\gamma_n|^{-1}  \sum_{j=2(m-1)}^\infty|\gamma_{c,j}|   \sum_{\mathcal{V} \in \Xi_{1^+,j,e_0}} \bigl| \Psi_\beta(\mathcal{V})\bigr|  
        +4\sum_{\mathcal{V} \in \Xi_{1^+,k_n+,e_0} } \bigl| \Psi_\beta(\mathcal{V}) \bigr|  
        \\&\qquad\leq  
        32 |\gamma_n|^{-1} \sum_{j=2(m-1)}^{\infty} j 
        \sum_{i=j}^\infty
         \binom{m}{2}2(i+1)^m C_{\beta^*} e^{-4(\beta-\beta^*)i} .
        \\&\qquad\qquad+
        4 \sum_{j=k_n}^\infty
            \binom{m}{2}2(j+1)^m C_{\beta^*} e^{-4(\beta-\beta^*)j} 
        .
    \end{align*} 
    Letting first \( N\) and then \( n\) tend to infinity, the desired conclusion immediately follows.
\end{proof}

\begin{proposition}\label{eq: proposition rectangle limit 2}
    Let \( \real  \beta > \beta_0(m).\) 
    Let  \(  R \geq 1, \) and let \( (T_n)_{n \ge 1}\) be a non-decreasing sequence of positive integers with \( T_n \geq R \) and \( \lim_{n \to \infty} T_n = \infty.\) For each \( n \geq 1,\) let \( \gamma_n\) be a rectangular loop with axis-parallel sides with lengths \( R\) and \( T_n\), respectively.
    Then the limit \( {\lim_{n\to \infty} -\log \langle W_{\gamma_n}\rangle_\beta/|\gamma_n|} \) exists and is given by
    \begin{equation*}
        \begin{split}
            \hat V_\beta(R) \coloneqq  V_\beta(R)/2 \coloneqq 
         &\sum_{k=1}^{R-1} \sum_{\mathcal{V} \in \Xi_{1^+,k,e_0}} \bigl| E_{\mathcal{V}}  \cap \support \gamma^0  \bigr|^{-1} \Psi_\beta(\mathcal{V}) \pigl( 1-\rho \bigl(\mathcal{V}(q^0)\bigr) \pigr) 
         \\&\qquad+
         \sum_{\mathcal{V} \in \Xi_{1^+,R+,e_0}} \bigl| E_{\mathcal{V}}  \cap \support \gamma^R  \bigr|^{-1} \Psi_\beta(\mathcal{V}) \pigl( 1-\rho \bigl(\mathcal{V}(q^R)\bigr) \pigr) .
         \end{split}
    \end{equation*}
    Moreover, for any \( n \geq 1, \) we have
    \begin{align*}
        \biggl| \frac{-\log \langle W_{\gamma_n} \rangle_\beta }{|\gamma_n|} 
        -
        \hat V_\beta(R)  
        \biggr|
        \leq& 
         32m^2|\gamma_n|^{-1} \sum_{j=2(m-1)}^\infty (3M)^{2j-1} (j+1)^{m+1}2^{j/(2(j-1))}e^{-4\beta j}
        \\&\qquad+ 4m^2\sum_{j=T_n}^\infty (3M)^{2j-1} (j+1)^m2^{j/(2(j-1))}e^{-4\beta j} = O_{m,\beta}(T_n^{-1}).
    \end{align*}
    Finally, we have
    \begin{align*} 
         \bigl| \hat V_\beta(R) - \hat V_\beta \bigr| \leq 4\sum_{j=R}^\infty (3M)^{2j-1} (j+1)^m2^{j}e^{-4\beta j}
    \end{align*}
    and hence \( \lim_{R \to \infty} \hat V_\beta(R) = \hat V_\beta.\)
\end{proposition}

\begin{proof}[Proof of Proposition~\ref{eq: proposition rectangle limit 2}]
    Note that for any \( k \geq 1 \) and \( N \) large enough to ensure that \(\dist(e_0,\partial B_N)>k,\) we have \( \Xi_{1^+,k,e_0}(B_N) = \Xi_{1^+,k,e_0}(B_k).\) 
    By Lemma~\ref{lemma: upper bound on term}, it follows that the sum 
    \begin{align*} 
         \hat V_{\beta,N}(R) \coloneqq &\sum_{k=1}^{R-1} \sum_{\mathcal{V} \in \Xi_{1^+,k,e_0}} \bigl| E_{\mathcal{V}}  \cap \support \gamma^0  \bigr|^{-1} \Psi_\beta(\mathcal{V}) \pigl( 1-\rho \bigl(\mathcal{V}(q^0)\bigr) \pigr) 
         \\&\qquad+ 
         \sum_{k=R}^\infty \sum_{\mathcal{V} \in \Xi_{1^+,k,e_0}} \bigl| E_{\mathcal{V}}  \cap \support \gamma^R  \bigr|^{-1} \Psi_\beta(\mathcal{V}) \pigl( 1-\rho \bigl(\mathcal{V}(q^R)\bigr) \pigr) 
    \end{align*}
    is well defined and absolute convergent, uniformly in \( N, \) and hence \( \hat V_\beta(R)\) is well defined.

    Fix \( n \geq 1. \)
    By Proposition~\ref{proposition: the cluster expansion}, we have
    \begin{equation*}
        \begin{split}
            &-\log \mathbb{E}[ W_{\gamma_n} ]_{\beta,N} 
            =
            \sum_{\mathcal{V} \in \Xi} \Psi_\beta(\mathcal{V}) \pigl( 1-\rho \bigl(\mathcal{V}(q_n)\bigr) \pigr)
            =
            \sum_{j=1}^{\infty}  \sum_{\mathcal{V} \in \Xi_{1^+,j}} \Psi_\beta(\mathcal{V}) \pigl( 1-\rho \bigl(\mathcal{V}(q_n)\bigr) \pigr)  .
        \end{split}
    \end{equation*}

    Let \( (k_n')_{n\geq 1}\) be a sequence of non-negative integers such that for each \( n \geq 1,\) \( {k_n' \leq \max(R_n,T_n),}\) and \( \lim_{n \to \infty} k_n'/|\gamma_n| = 0.\) Then, for each \( n \geq 1, \) by applying Lemma~\ref{lemma: general limit bound 0} with \(k=k'=k_n ,\) we obtain 
    \begin{align*}
        &\biggl| -\log \mathbb{E}[ W_{\gamma_n} ]_{\beta,N} 
        -
        |\gamma_n|\hat V_{\beta,N}(R)  
        \biggr|
        \\&\qquad\leq 
        4 \sum_{j= 1}^{k_n'-1}|\gamma_{c,j}|  \sum_{\mathcal{V} \in \Xi_{1^+,j,e_0}} \bigl| \Psi_\beta(\mathcal{V}) \bigr|   
        +
        4|\gamma_n|\sum_{\mathcal{V} \in \Xi_{1^+,k_n'+,e_0} } \bigl| \Psi_\beta(\mathcal{V}) \bigr| 
    \end{align*}
    Proceeding as in the proof of Proof of Proposition~\ref{eq: proposition rectangle limit 1}, we obtain the desired conclusion.
\end{proof}

\begin{proof}[Proof of Theorem~\ref{theorem: limit of ratio exists 2}]
    By Proposition~\ref{eq: proposition rectangle limit 2} applied with \( T_n = n, \)  \( T_n = R, \) the limit 
    \[ 
        \lim_{T \to \infty} - \frac{1}{T} \log \, \langle W_{\gamma_{R,T}} \rangle_\beta = \lim_{ T \to \infty } \lim_{N \to \infty}
        -\frac{1}{T}{\log \mathbb{E}[ W_{\gamma_n} ]_{\beta,N}  }
        =
        \lim_{ n \to \infty } \lim_{N \to \infty}
        -\frac{2}{{|\gamma_n|}}{\log \mathbb{E}[ W_{\gamma_n} ]_{\beta,N}   }
    \] 
    exists and is equal to \(2V_\beta(R). \)
    Using Theorem~\ref{theorem: main theorem}, the desired conclusion follows.
\end{proof}

\begin{proof}[Proof of Theorem~\ref{theorem: limit of ratio exists}]
    By Proposition~\ref{eq: proposition rectangle limit 1} applied with \( R_n = Hn \) and \( T_n = Ln,\) the limit 
        \[ 
   \lim_{n \to \infty} -\frac{1}{R_n+T_n} \log \, \langle W_{\gamma_{n}} \rangle_\beta = \lim_{n \to \infty} \lim_{N \to \infty}
    -\frac{2}{{|\gamma_{n}|}}{\log \langle W_{\gamma_{n}} \rangle_{N,\beta} } .
    \] 
    exists and is equal to \( 2 \hat V_\beta.\)  We conclude using Theorem~\ref{theorem: main theorem}.
\end{proof}

\end{document}